\documentclass[11pt]{article}
\usepackage{amsfonts,latexsym,rawfonts,amsmath,amssymb,amsthm,a4wide, times, color}
\usepackage{graphicx}
\usepackage[plainpages=false]{hyperref}
\numberwithin{equation}{section}
\usepackage{hyperref}
\usepackage[titletoc,title]{appendix}

 \usepackage[all]{xy}
 \usepackage{eufrak}
 \usepackage{makeidx}
 \usepackage{graphicx,psfrag}

\numberwithin{equation}{section}

\newcommand{\pd}[2]{\frac {\partial #1}{\partial #2}}

\newcommand{\al}{\alpha}
\newcommand{\bb}{\beta}
\newcommand{\la}{\lambda}
\newcommand{\La}{\Lambda}
\newcommand{\oo}{\omega}
\newcommand{\Om}{\Omega}

\newcommand{\dd}{\delta}
\newcommand{\Na}{\nabla}

\def\ga{\gamma}

\newcommand{\ee}{\epsilon}

\newcommand{\Si}{\Sigma}
\newcommand{\Te}{\Theta}
\newcommand{\te}{\theta}

\newcommand{\beq}{\begin{equation}}
\newcommand{\eeq}{\end{equation}}
\newcommand{\beqs}{\begin{eqnarray*}}
\newcommand{\eeqs}{\end{eqnarray*}}
\newcommand{\beqn}{\begin{eqnarray}}
\newcommand{\eeqn}{\end{eqnarray}}
\newcommand{\beqa}{\begin{array}}
\newcommand{\eeqa}{\end{array}}

\def\td{\tilde}
\def\p{\partial}
\def\b{\backslash}

\def\RR{{\mathbb R}}
\def\NN{{\mathbb N}}

\def\ri{\rightarrow}

\def\vol{{\rm Vol}}
\def\Ann{{\rm Ann}}

\def\cA{{\mathcal A}}

\def\cP{{\mathcal P}}

\def\cS{{\mathcal S}}

\def\ka{{\kappa}}

\def\x{{\mathbf{x}}}

\def\n{{\mathbf{n}}}

\def\m{{\frak m}}

\def\Area{{\mathrm {Area}}}
\def\Supp{{\mathrm {Supp}}}

\renewcommand\div{{\rm div}}
\newtheorem{prop}{Proposition}[section]
\newtheorem{theo}[prop]{Theorem}
\newtheorem{lem}[prop]{Lemma}
\newtheorem{claim}[prop]{Claim}
\newtheorem{cor}[prop]{Corollary}
\newtheorem{rem}[prop]{Remark}

\newtheorem{defi}[prop]{Definition}

\title{ The extension problem of the mean curvature flow (I)}
\author{Haozhao Li \footnote{Supported by NSFC grant No. 11671370 and the Fundamental Research Funds for the Central Universities.} \quad  and \quad  Bing Wang \footnote{Supported by NSF grant DMS-1510401.}}

\begin{document}
\bibliographystyle{plain}


\maketitle

\begin{abstract}
We show that the mean curvature  blows up at the first finite singular time for a closed smooth  embedded mean curvature flow in $\RR^3.$
\end{abstract}

\tableofcontents

\section{Introduction}
Let $\x_0: \Si^n\ri \RR^{n+1}$ be a closed smooth embedded hypersurface in $\RR^{n+1}. $ A one-parameter family of immersions $\x(p, t): \Si^n\ri \RR^{n+1}$ is called a mean curvature flow, if $\x$  satisfies the equation
\beq
\pd {\x}t=-H\n,\quad \x( 0)=\x_0, \label{eq:MCF0}
\eeq where $H$ denotes the mean curvature of the hypersurface $\Si_t:=\x(t)(\Si)$ and $\n$ denotes the unit normal vector field of $\Si_t$.  In \cite{[Hui1]}, Huisken proved that if the flow (\ref{eq:MCF0}) develops a singularity at time $T<\infty$, then the fundamental form will blow up at time $T$. A natural conjecture is that the mean curvature will blow up at the finite singular time of a mean curvature flow. There are many results toward this conjecture(c.f. Cooper~\cite{[Cooper2]}, Le-Sesum~\cite{[LS]}\cite{[LS2]}\cite{[LS2a]}, Lin-Sesum~\cite{[LS3]}, Xu-Ye-Zhao~\cite{[XYZ]} ...).
 This conjecture is also proposed in page 42 of Mantegazza's book \cite{[Man]} as an open problem. In this paper, we confirm this conjecture in dimension two:

\begin{theo}\label{theo:main1} If
 $\x(p, t): \Si^2\ri \RR^3 (t\in [0, T))$ is a closed smooth embedded mean curvature flow
with the first  singular time $T<+\infty$, then
$$\sup_{\Si\times [0, T)}|H|(p, t)=+\infty.$$

\end{theo}

It is interesting to compare the extension problem  of mean curvature flow with Ricci flow. In \cite{[Ha1]} Hamilton proved that the Riemann curvature tensor will blow up at the finite singular time of a Ricci flow. In \cite{[Sesum1]} Sesum extended Hamilton's result to the Ricci curvature by using Perelman's noncollapsing theorem. In a series of papers \cite{[CW-extend1]}\cite{[CW-extend2]}\cite{[CW-extend3]} Wang and Chen-Wang provided several conditions which can be used to extend Ricci flow, and they  gave a different proof of Sesum's result. It is also conjectured that the Ricci flow can be extended if the scalar curvature stays bounded. Important progresses have been made by Zhang \cite{[Zhang]},  Bamler-Zhang \cite{[BZ]} and Simon \cite{[MSimon]}.

The mean curvature flow with convexity conditions has been well studied during the past several decades. If the initial hypersurface satisfies some convexity conditions, like mean convex or two convex, then the mean curvature flow has some convexity estimates (c.f.  Huisken \cite{[Hui1]}, Huisken-Sinestrari \cite{[HS99a]}\cite{[HS99b]}, Haslhofer-Kleiner \cite{[HK13]}). These convexity estimates are important for studying the surgery of mean curvature flow (c.f. Huisken-Sinestrari\cite{[HS09]},  Brendle-Huisken \cite{[BH]}, Haslhofer-Kleiner \cite{[HK14]}). In  \cite{[White00]} \cite{[White03]} White also gave some important properties of the singularities of a mean curvature flow with mean convex initial hypersurfaces. However, all these results rely on the convexity condition of  initial hypersurfaces, and it is very difficult to study  general cases (c.f. Colding-Minicozzi~\cite{[CM2]}~\cite{[CM3]}, T. Ilmanen~\cite{[Ilmanen]}). Theorem \ref{theo:main1} can be viewed as an attempt to study the general singularities without assuming convexity conditions.\\

Now we sketch the proof of Theorem \ref{theo:main1}. Assume that the mean curvature is  bounded along the flow (\ref{eq:MCF0}) and the first singular time $T<+\infty$. Consider the
corresponding rescaled mean curvature flow
\beq \Big(\pd {\x}t\Big)^{\perp}=-\Big(H-\frac 12\langle\x, \n  \rangle\Big)\n,\quad \forall\;t\in [0, \infty).\label{eq:RMCF0}\eeq
Then the mean curvature decays exponentially to zero along the flow (\ref{eq:RMCF0}).
We have to show that  the  flow (\ref{eq:RMCF0}) converges smoothly to a plane with multiplicity one. The proof consists of two steps:\\

\emph{ Step 1. Convergence of the rescaled mean curvature flow with multiplicities.}

In this step,   we first follow the ideas of Chen-Wang \cite{[CW1]}  \cite{[CW2]} to develop the weak compactness theory of mean curvature flow with certain properties,
which are basically area doubling property together with bounded mean curvature $H$ and bounded energy $\int_{\Sigma} |A|^2$.
Note that the energy $\int_{\Sigma} |A|^2$ can be bounded by $H$ and topology, via the Gauss-Bonnett theorem(c.f. (\ref{eqn:PE10_1})).
To prove such weak compactness, there are two main technical ingredients:
\begin{itemize}
  \item  The two-sided, long-time pseudolocality theorem.
  \item The energy concentration property.
\end{itemize}
The short-time forward pseudolocality theorem for mean curvature flow was studied under the local graphic condition, which says that  the surface can be written locally as a graph of a single-valued function.
See  Ecker-Huisken~\cite{[EH1]} \cite{[EH2]}, M.T. Wang~\cite{[WangMT]}, Chen-Yin~\cite{[ChenYin]} and S.Brendle~\cite{[BH]}.
In our case, we need to remove this graphic condition since  the flow may converge with multiplicities.
Under the assumption that the mean curvature is bounded, we will show a new short-time, two-sided pseudolocality type theorem(c.f. Theorem~\ref{theo:pseudoA}).
Moreover, the pseudolocality theorem can be improved as long-time version(c.f. Theorem~\ref{theo:pseudoB}) whenever the mean curvature is very small.
Then the energy concentration property follows from the pseudolocality theorem.
Once we have the long-time, two-sided pseudolocality theorem  and the energy concentration property (c.f. Lemma~\ref{lem:energy concen1}), we can show the weak compactness of mean curvature flow and get the ``flow" convergence of the rescaled mean curvature flow.
Since the limit is both minimal and self-shrinking, we obtain the limit must be a plane passing through the origin, with possibly more than one multiplicity.
\\

\emph{ Step 2. Show that the multiplicity of the convergence is one.}

In this step, we use the ideas from the compactness of self-shrinkers and minimal surfaces (c.f. Colding-Minicozzi~\cite{[CM4]}\cite{[CM1]}, L.Wang~\cite{[LWang]}) to show that the multiplicity is one.
For otherwise we can construct a sequence of positive solutions to the  corresponding  parabolic equations on compact sets of the limit surface away from singularities.
However, we shall show that the existence of such a solution contradicts the fact that the limit is a plane with multiplicities.
To obtain the contradiction, there are also two main technical ingredients:
\begin{itemize}
  \item  Uniform estimates of the sequence of positive solutions.
  \item The $L$-stability of the limit surface across the singular set.
\end{itemize}
To show uniform estimates of these positive solutions, we introduce a new ``almost" decreasing quantity along the rescaled mean curvature flow, and use this quantity with the parabolic Harnack inequality to control the positive solutions. After normalization and taking the limit, we get a positive solution with good estimates on the limit surface away from singularities. Using these estimates, we show that the limit surface is $L$-stable away from singularities. Recall that a hypersurface $\Si$
is called $L$-stable(c.f. Colding-Minicozzi~\cite{[CM1]}) if for any compactly supported function $u$, we have
$$\int_{\Si}\, -uLu \,e^{-\frac {|x|^2}{4}}\geq 0,$$
where $L$ is the operator
$Lu=\Delta u+|A|^2u-\frac 12\langle x, \Na u\rangle+\frac u2$ defined by Colding-Minicozzi~\cite{[CM2]}.
Following the argument
of Gulliver-Lawson \cite{[GL]} in minimal surfaces, we show that the limit surface is actually $L$-stable across the singular set.
However, the limit surface is a plane which is not $L$-stable. Thus we obtain the desired contradiction. This contradiction forces that the convergence of (\ref{eq:RMCF0}) must have multiplicity one and hence in the smooth topology.\\

Once   the rescaled mean curvature flow (\ref{eq:RMCF0}) converges smoothly to a plane with multiplicity one,  we use Huisken's monotonicity formula and  White's regularity theorem to show that the mean curvature flow (\ref{eq:MCF0}) actually has no singular points at time $T$. This finishes the proof of Theorem \ref{theo:main1}.\\

The strategy of the proof of Theorem~\ref{theo:main1} is very similar to the one used in the study of the K\"ahler Ricci flow on Fano manifolds by Chen-Wang~\cite{[CW2]}(See~\cite{[CW2A]} and~\cite{[CW2B]} for the published version)
and Chen-Sun-Wang~\cite{[CSW]}.
Actually, in Chen-Wang~\cite{[CW2]}, the flow weak-compactness was setup for the normalized K\"ahler Ricci flow, which is comparable to the step 1 mentioned above.
Then in Chen-Sun-Wang~\cite{[CSW]}, the limit and convergence topology can be improved to be ``smooth" whenever the underlying manifold is $K$-stable.  This is related to our step 2 described above.
However, here we used the $L$-stability of the limit surfaces, instead of the $K$-stability of the the underlying manifolds to rule out the possible singularities.
Not surprisingly, the strategy can also be applied to study the extension problem for the 4-dimensional Ricci flow with bounded scalar curvature.
Actually,  based on the calculation of M. Simon(c.f.~\cite{[MSimon]}), we have uniform bounded energy $\int_{M} |Rm|^2 $ along the flow.
Then the  weak-compactness of the parabolic normalized Ricci flows(compared with equation (\ref{eq:RMCF0}))
$$\partial_t g=-2(Ric-g)$$
follows from Chen-Wang~\cite{[CW1]}(See also~\cite{[BZ]} for a different approach).
The limit of the above flow is then a Ricci-flat gradient shrinking soliton with finite singular points, which is nothing but a flat metric cone over $S^3/\Gamma$ for some finite subgroup $\Gamma$ of $O(4)$.
Similar to the step 2 mentioned above, the 4-dimensional Ricci flow extension problem will be confirmed if one can develop methods to rule out the nontrivial $\Gamma$'s. \\

We remark that the weak compactness of the mean curvature flow, i.e., step 1, is based on the observation that locally the structure of the mean curvature flow with bounded $H$ and energy $\int_{\Sigma}\, |A|^2 $
is modeled by the structure of minimal surfaces with uniformly bounded topology.  Similar observation is also the key for the weak compactness of the Ricci flow with bounded  energy $\int_{M} \,|Rm|^{\frac{m}{2}} $ and bounded scalar curvature $R$.
The two-sided,  long-time pseudo-locality(c.f. Chen-Wang~\cite{[CW1]},~\cite{[CW2]} and~\cite{[CW3]}) and the energy concentration are the technical tools for writing down the observation in rigorous analysis estimates.
It seems that such observation holds for many other geometric flows, e.g., the harmonic map flow and the Calabi flow. \\

The organization of this paper is as follows. In Section 2 we recall some basic facts on mean curvature flow and minimal surfaces. In Section 3 we develop the weak compactness theory of mean curvature flow under some geometric conditions. In Section 4 we show the rescaled mean curvature flow with the exponential decay of mean curvature  converges smoothly to a plane with multiplicity one.  Finally, we finish the proof of Theorem \ref{theo:main1} in Section 5. \\

 {\bf Acknowledgements}: H. Z. Li would like to thank Professors T.H. Colding, W. P. Minicozzi II and X. Zhou for insightful discussions.  Part of this work was done while he was visiting MIT and he wishes to thank MIT for their generous hospitality.
 B. Wang would like to thank Professors T. Ilmanen,  L. Wang and O. Hershkovits  for helpful discussions.
 Both authors are grateful to  the anonymous referees for many useful suggestions to improve the exposition of this paper.

\section{Preliminaries}

Let $\x(p, t): \Si^n\ri \RR^{n+1}$ be a family of smooth  embeddings  in $\RR^{n+1}$. $\{(\Si^n, \x(t)), 0\leq t< T\}$ is called a  mean curvature flow if $\x(t)$ satisfies  \beq \pd {\x}t=-H\n,\quad \forall\; t\in [0, T),\label{eq:MCF}\eeq
and called a rescaled mean curvature flow if $\x(t)$ satisfies
 \beq \Big(\pd {\x}t\Big)^{\perp}=-\left(H-\frac 12\langle \x, \n\rangle \right)\n,\quad \forall\; t\in [0, T),\label{eq:RMCF}\eeq
Sometimes we also write $\x_t$  as $\x(t)$ for short.
It is easy to check that the rescaled mean curvature flow is equivalent to the mean curvature flow after rescalings in space
and a reparameterization of time. We denote by $B_r(p)$ the open  ball in $\RR^3$ centered at $p$ with radius $r$.

It is well known that the volume ratio is bounded from below along the flow (\ref{eq:MCF}). See, for example,  Lemma 2.9 of Colding-Minicozzi \cite{[CM2]}.
\begin{lem} \label{lem:vol} Let $\{(\Si^n, \x(t)), 0\leq t< T\}$ be a mean curvature flow (\ref{eq:MCF}). Then there is a constant $N=N(\vol(\Si_0), T)>0$ such that for all $r>0$ and $p_0\in \RR^{n+1}$ we have
$$\vol(B_r(p_0)\cap \Si_t)\leq N r^n, \quad \forall\; t\in [0, T).$$
\end{lem}

A hypersurface $\x: \Si^n\ri \RR^{n+1}$ is called a self-shrinker, if it satisfies the equation \beq H=\frac 12 \langle \x, \n \rangle.\label{eq:Selfshrinker}\eeq
By Corollary 2.8 of Colding-Minicozzi \cite{[CM2]}, we have
\begin{lem}\label{lem:selfshrinker}(Corollary 2.8 of \cite{[CM2]}) If $\Si$ is a self-shrinker and the mean curvature is zero, then $\Si$ is a minimal cone. In particular, if $\Si$ is also smooth and embedded, then it is a hyperplane through $0$.
\end{lem}

Observe that the equation (\ref{eq:MCF}) is invariant under the rescaling
\beq
\td \x(p, s)=\la \Big(\x(p, T+\frac {s}{\la^2})-p_0\Big),\quad
\forall \;(p, s)\in \Si\times[-T\la^2, 0), \label{eq:D001}
\eeq where $\la>0$. In fact, under the rescaling (\ref{eq:D001}) we have the relations
$$\td A_{ij}=\la A_{ij}, \quad \td g_{ij}=\la^2 g_{ij},\quad
\td H=\frac 1{\la}H, \quad |\td A|=\frac 1{\la}|A|, $$
where $\td A, \td g$ and $\td H$ denote the second fundamental form, the induced metric and the mean curvature of the surface $\td \Si=\td \x(\Si)$ respectively.
Moreover, by direct calculation we have

\begin{lem} \label{lem:invariant} Let $\Si^2$ be a smooth surface in $\RR^3$. The area ratio and the $L^2$ norm of the second fundamental form are invariant under (\ref{eq:D001}).
Namely, we have  $$\frac {\Area_{g}(B_r(p)\cap \Si)}{\pi r^2}=\frac {\Area_{\td g}(B_{\td r}(\td p)\cap \td \Si)}{\pi {\td r}^2}$$
and
$$\int_{B_r(p)\cap \Si}\,|A|^2\,d\mu=\int_{B_{\td r}(\td p)\cap \td \Si}\,|\td A|^2\,d\td \mu,$$ where
$\td r=\la r,\;\td p=\la(p-p_0) $ and $\td \Si=\la(\Si-p_0).$
\end{lem}

Now we recall some facts on compactness of immersed hypersurfaces in $\RR^{n+1}$. For any $B_R(p)\subset \RR^{n+1}$, we say an immersed (or embedded) hypersurface $\Si^n\subset \RR^{n+1}$ properly immersed (or embedded) in $B_R(p)$, if either $\Si$ is closed or $\p \Si$ has distance at least $R$ from the point $p$. We say that a sequence $\{\Si_i\}$ converges in $C_{\mathrm{loc}}^{k, \al}$ topology to a hypersurface $\Si$, if for any $p\in \Si$ each $\Si_i$ is locally (near $p$) a graph over the tangent space $T_p\Si$ and the graph of $\Si_i$ converges to the graph of $\Si$ in the usual $C_{\mathrm{loc}}^{k, \al}$ topology.

The following theorem is well known and it shows the convergence of properly immersed surfaces under the second fundamental form bound. The readers are referred to \cite{[White1]} for the details.

\begin{theo}[\textbf{Compactness of minimal surfaces}]\label{theo:minimalcompactness} Let $\{\Si_i^n\}$ be a sequence of smooth, properly immersed hypersurfaces in $B_R(p)\subset \RR^{n+1}.$
If the second fundamental form $\sup_{\Si_i\cap B_R(p)}|A_i|\leq \La$ for a constant $\La>0$ and all $i\geq 1$, then there is a subsequence of $\{\Si_i^n\}$ converges in $C^{1, \frac 12}$
topology, possibly with multiplicities, to an immersed hypersurface $\Si_{\infty}$.
Moreover, if the hypersurfaces $\{\Si_i^n\}$ are minimal, then the convergence is in smooth topology.
\end{theo}

To study the convergence of mean curvature flow, we define the convergence of a sequence of one-parameter hypersurfaces as follows.

\begin{defi}
We say that a sequence of one-parameter smoothly immersed hypersurfaces $\{\Si^n_{i, t}, -1<t<1\}$ in $\RR^{n+1}$ converges in smooth topology, possibly with multiplicities, to a limit flow $\{\Si_{\infty, t}, -1<t<1\}$ away from a space-time singular set $\cS\subset \RR^{n+1}\times (-1, 1)$, if for any $t\in (-1, 1)$, any $p\in \Si_{\infty, t}\b\cS_t$ and large $i$, there exists $r>0$ and $\ee>0$ such that the hypersurface $\Si_{i, s}\cap B_r(p)$ with $s\in [t-\ee, t+\ee]$  can be  written as a collection of graphs of smooth functions $\{u_i^1(x, s), u_i^2(x, s), \cdots, u_i^N(x, s)\}$ over the tangent plane of $\Si_{\infty, t}$ at the point $p$. Moreover, for each $k\in \{1, 2, \cdots, N\}$ the functions $u_i^k(x, s)$ converges smoothly in $x$ and $s$ as $i\ri +\infty$.

In the above definition,  $\cS_t$ is defined by
 $\cS_t=\{x\in \RR^{n+1}\,|\,(x, t)\in \cS\}$. If $\cS_t$ is independent of $t$, then we can also replace the space-times singular set $\cS$ simply  by $\cS_{t_0}$ for some $t_0.$

\end{defi}

The following compactness result of mean curvature flow is well-known. See, for exmaple, page 481-482 of \cite{[ChenYin]} for a detailed proof.

\begin{theo}[\textbf{Compactness of mean curvature flow}]\label{theo:MCFcompactness}Let $\{(\Si_i^{n}, \x_i(t)), -1< t< 1 \}$ be a sequence of
mean curvature flow properly immersed in $B_R(0)\subset \RR^{n+1}$(i.e. for each $t\in (-1, 1)$ the hypersurface $\Si_{i, t}$ is properly immersed in $B_R(0)$).
Suppose that $$\sup_{\Si_{i, t}\cap B_R(0)}|A|(x, t)\leq \La,\quad \forall\; t\in (-1, 1)$$ for some $\La>0.$
 Then   a subsequence of $\{\Si_{i, t}\cap B_R(0), -1< t< 1 \}$ converges in smooth topology to a smooth  mean curvature flow $\{\Si_{\infty, t}, -1<t<1\}$ in $B_R(0)$.
\end{theo}

\section{Weak compactness of the mean curvature flow}

In this section, we follow the arguments of Ricci flow by Chen-Wang in \cite{[CW1]}\cite{[CW2]} and minimal surfaces by Choi-Schoen in \cite{[CS]} to study the weak compactness of mean curvature flow under some geometric conditions. This weak compactness result will be used to prove the convergence of rescaled mean curvature flow in the next section.

\subsection{The pseudolocality theorem}
The pseudolocality type results of the  mean curvature flow  were  studied by Ecker-Huisken  \cite{[EH1]} \cite{[EH2]},   M. T. Wang \cite{[WangMT]}, Chen-Yin \cite{[ChenYin]} and Brendle-Huisken \cite{[BH]}.
However, all these pseudolocality theorems above require the condition that the initial hypersurface can be locally written as a graph of a single-valued function.  In our case, we need to remove this graphic condition since  the flow may converge with multiplicities. Here, we give a different type of pseudolocality theorem under the assumption that the mean curvature along  the flow is uniformly bounded.

\begin{defi}
For any $r>0, p\in \RR^{n+1}$ and $\Si^n\subset \RR^{n+1}$, we denote by {$C_x(B_r(p)\cap \Si)$} the connected component of $ B_r(p)\cap \Si$ containing $x\in \Si. $
\end{defi}

\begin{lem}\label{lem:graph1}(c.f. Lemma 7.1 of \cite{[ChenYin]}) Let $\Si^n\subset  \RR^{n+1}$ be properly embedded in $B_{r_0}(x_0)$ for some $x_0\in \Si$ with
$$|A|(x)\leq \frac 1{r_0}\,\quad x\in B_{r_0}(x_0)\cap \Si. $$
Let $\{x^1\, \cdots, x^{n+1}\}$ be the standard coordinates in $\RR^{n+1}$.
Assume that $x_0=0$ and the tangent plane of $\Si$ at $x_0$ is $x^{n+1}=0.$ Then there is a map
$$u: \Big\{x'=(x^1, \cdots, x^n)\,\Big|\, |x'|<\frac {r_0}{96}\Big\}\ri \RR$$
with $u(0)=0$ and $|\Na u|(0)=0$ such that the connected component containing $x_0$ of $\Si\cap
\{(x', x^{n+1})\in \RR^{n+1}\;|\;|x'|<\frac {r_0}{96}\}$  can be written as a graph $\{(x', u(x'))\,|\,|x'|<\frac {r_0}{96}\}$ and
$$|\Na u|(x')\leq \frac {36}{r_0}|x'|.$$
\end{lem}

Using Lemma \ref{lem:graph1}, we show that the local area ratio of the surface is very close to $1$.

\begin{lem}\label{lem:ratio}Suppose that $\Si^n\subset B_{r_0}(p)\subset\RR^{n+1}$ is a hypersurface with $\p \Si\subset \p B_{r_0}(p)$ and
$$ \sup_{\Si}\,|A|\leq \frac 1{r_0}.$$
For any $\dd>0$, there is a constant $\rho_0=\rho_0(r_0, \dd)$ such that for any $r\in (0, \rho_0)$ and any $x\in B_{\frac {r_0}{2}}(p)\cap \Si$ we have
 \beq
\frac {\vol_{\Si}(C_x(B_r(x)\cap \Si))}{\oo_n r^n}\leq 1+\dd.\label{eq:G007}
\eeq

\end{lem}
\begin{proof}By Lemma \ref{lem:graph1}, for any $x\in B_{\frac {r_0}{2}}(p)\cap \Si$ the component $C_x(B_{\rho_0}(x)\cap \Si)$ with $\rho_0=\frac {r_0}{192}$ can be written as a graph of a function $u$ over the tangent plane at $x$, which we assume to be $P=\{(x_1, \cdots, x_n, x_{n+1})\in \RR^{n+1}\;|\; x_{n+1}=0\}$,   with $|\Na u|(x')\leq \frac {72}{r_0}|x'|$ where $x'=(x_1, \cdots, x_n).$   For any $r\in (0, \rho_0)$,  the volume ratio of $C_x(B_r(x)\cap \Si)$ is given by
$$\frac {\vol_{\Si}(C_x(B_r(x)\cap \Si))}{\oo_n r^n}\leq\frac 1{\oo_n r^n}\int_{B_r(x)\cap P}\,\sqrt{1+|\Na u|^2}\,d\mu\leq \sqrt{1+\frac {5184}{r_0^2}r^2}.  $$
Thus, we can choose $r$ sufficiently small such that (\ref{eq:G007}) holds. The lemma is proved.

\end{proof}

The next result shows that we can control the local volume ratio along the mean curvature flow  with bounded mean curvature.

\begin{lem}\label{lem:vol2}  Let $\{(\Si^n, \x(t)), -T\leq t\leq T\}$ be a smooth embedded mean curvature flow (\ref{eq:MCF}) with
$
\max_{\Si_t\times [-T, T]}|H(p, t)|\leq \La.
$
   Then for any $t_1, t_2\in [-T, T]$ we have
\beq
\frac {\vol_{g(t_2)}(C_{p_{t_2}}(B_{r_2}(p_{t_2})\cap \Si_{t_2}))}{\oo_n r_2^n}
\leq f(t_1, t_2, \La, r_2) \frac {\vol_{ g(t_1)}(C_{ p_{t_1}}(B_{ r_1}( p_{t_1})\cap \Si_{t_1}))}{\oo_n  r_1^n} \label{eq:K005}
\eeq where $p_t=\x_t(p, t)$ for some $p\in \Si$ and
\beqs
r_1&=&r_2+2\La |t_2-t_1|,\\
f(t_1, t_2, \La, c, r_2)&=& e^{\La^2|t_2-t_1|} \Big(1+\frac {2\La }{r_2}|t_2-t_1|\Big)^n
\eeqs

\end{lem}
\begin{proof}Since the mean curvature is bounded, for any $\Om\subset \Si$ we calculate
\begin{align}
\Big|\frac d{dt}\vol_{ g(t)}(\Om)\Big|=\Big|\int_{\Om}\,| H|^2\,d \mu_t\Big|\leq \La^2\vol_{g(t)}(\Om),  \label{eqn:PE24_1}
\end{align}
which implies that
\beq e^{-\La^2 |t_1-t_2|}\vol_{g(t_2)}(\Om)\leq \vol_{ g(t_1)}(\Om)\leq e^{\La^2|t_1-t_2|}\vol_{g(t_2)}(\Om). \label{eq:A205a}\eeq
 Note that for any $p, q\in \Si,$ we have \beqs   -4\La |  \x(p, t)-  \x(q, t)|&\leq &
\pd {}t|  \x(p, t)-  \x(q, t)|^2\\&=&2\langle   \x(p, t)-  \x(q, t), -  H(p, t)\n(p, t)+  H(q, t)\n(q, t)\rangle\\
&\leq &4\La |  \x(p, t)-  \x(q, t)|.
\eeqs
Therefore, we have
$$-2\La\leq \pd {}t|\x(p, t)-  \x(q, t)|\leq 2\La. $$
It follows that
\beq |  \x(p, t_1)-  \x(q, t_1)|\leq |  \x(p, t_2)-  \x(q, t_2)|+2\La |t_1-t_2|.  \label{eqG:001}\eeq
Let $  p_t=  \x(p, t)$. Then we have
$$  \x(t_2)^{-1}(C_{  p_{t_2}}(B_r(  p_{t_2})\cap   \Si_{t_2}))\subset
  \x(t_1)^{-1}(C_{  p_{t_1}}(B_{r+2\La |t_1-t_2|}(  p_{t_1})\cap   \Si_{t_1})).$$
Therefore, we have the  estimates
\beqn &&
\vol_{  g(t_2)}(C_{  p_{t_2}}(B_r(  p_{t_2})\cap   \Si_{t_2}))\nonumber\\
&\leq& \vol_{  g(t_2)}(\x(t_1)^{-1}(C_{  p_{t_1}}(B_{r+2\La e |t_1-t_2|}(  p_{t_1})\cap   \Si_{t_1})))\nonumber\\
&\leq &e^{ \La^2 |t_1-t_2|}\vol_{  g(t_1)}(C_{  p_{t_1}}(B_{r+2\La |t_1-t_2|}(  p_{t_1})\cap   \Si_{t_1})),
\eeqn where we used  (\ref{eq:A205a}) in the last inequality.
For any $x\in \Si^n\subset \RR^{n+1}$ and $r_2>0$, we have
\beqs &&
\frac {\vol_{g(t_2)}(C_{p_{t_2}}(B_{r_2}(p_{t_2})\cap \Si_{t_2}))}{\oo_n r_2^n}\\&\leq&e^{ \La^2 |t_1-t_2|}\cdot \Big(\frac {  r_1}{  r_2}\Big)^n \cdot \frac {\vol_{  g(t_1)}(C_{  p_{t_1}}(B_{  r_1}(  p_{t_1})\cap   \Si_{t_1}))}{\oo_n r_1^n}\\
&=&e^{ \La^2 |t_1-t_2|}\cdot \Big(\frac {  r_1}{  r_2}\Big)^n \cdot \frac {\vol_{ g(t_1)}(C_{ p_{t_1}}(B_{ r_1}( p_{t_1})\cap \Si_{t_1}))}{\oo_n r_1^n},
\eeqs where $r_1, r_2$ satisfy the following relations:
$$
 r_1=r_2+2\La |t_1-t_2|.
$$
The lemma is proved.

\end{proof}

Using the idea of the monotonicity formula of minimal surfaces, we show that the volume ratio is almost monotone if the mean curvature is bounded. See, for example, Proposition 1.12 in Colding-Minicozzi \cite{[CMbook2]}.

\begin{lem}\label{lem:vol3}

Let $\Si^n\subset \RR^{n+1}$ be a properly embedded hypersurface in $B_{r_0}(x_0)$ with $x_0\in \Si$ and $|H|\leq \La$. Then for any $s\in (0, r_0)$ we have
\beq
\frac {\vol_{\Si}(B_s(x_0)\cap \Si)}{\oo_n s^n}\leq e^{\La r_0}\cdot \frac {\vol_{\Si}(B_{r_0}(x_0)\cap \Si)}{\oo_n {r_0}^n}.\nonumber
\eeq
In particular, letting $s\ri 0$  we have
$$\vol_{\Si}(B_{r}(x_0)\cap \Si)\geq e^{-\La r}\oo_n r^n,\quad \forall \,r\in (0, r_0].$$
\end{lem}

\begin{proof}

Note that the function $f(x)=|x-x_0|$ satisfies the identity
$$\Delta_{\Si}f^2=2n-2H \langle x-x_0,\n \rangle.$$
By the Stokes' theorem, we have
\beqn
2n\vol(\{f\leq s\})&=&\int_{\{f\leq s\}}\,\Delta_{\Si}\,f^2+2\int_{\{f\leq s\}}\,H \langle x-x_0,\n \rangle\nonumber \\
&=&2\int_{\{f= s\}}\,|(x-x_0)^T|+2\int_{\{f\leq s\}}\,H \langle x-x_0,\n \rangle.\label{eq:H002}
\eeqn
The coarea formula implies that
\beq \vol(\{f\leq s\})=\int_0^s\,\int_{\{f=r\}}\, |\Na_{\Si}f|^{-1}.\label{eq:H003}\eeq
Combining the identities (\ref{eq:H002})-(\ref{eq:H003}), we have
\beqs
\frac d{ds}\,\Big(s^{-n}\vol(\{f\leq s\})\Big)&=&-ns^{-n-1}\vol(\{f\leq s\})
+s^{-n}\int_{\{f= s\}}\,\frac {|x-x_0|}{|(x-x_0)^T|}\\
&=&s^{-n-1}\int_{\{f= s\}}\,\frac {|(x-x_0)^N|^2}{|(x-x_0)^T|}-s^{-n-1}\int_{\{f\leq s\}}\,H \langle x-x_0,\n \rangle\\
&\geq&-\La \cdot s^{-n}\vol(\{f\leq s\}).
\eeqs
Let $F(s)=s^{-n}\vol(\{f\leq s\}).$ Then for any $s\in (0, r_0)$ we have
$$F(s)\leq F(r_0)e^{\La (r_0-s)}\leq F(r_0)e^{\La r_0}. $$
Thus, the lemma is proved.

\end{proof}

\begin{lem}\label{lem:gap} For any $C>0$, there exists  $\dd=\dd(n, C)>0$ satisfying the following property. Any complete smooth minimal hypersurface $\Si^n\subset \RR^{n+1}$ with bounded second fundamental form $|A|\leq C$ and volume ratio
\beq
\frac {\vol_{\Si}(B_r(p)\cap \Si)}{\oo_n r^n}<1+\dd,\quad \forall \; r>0
\eeq must be a hyperplane.
\end{lem}
\begin{proof}Suppose not, there exists a sequence of non-flat minimal hypersurfaces $\Si_i$ with  $|A_i|\leq C$ and
\beq
\frac {\vol(B_r(p_i)\cap \Si_i)}{\oo_n r^n}<1+\dd_i,\quad \forall \; r>0
\eeq where $p_i\in \Si_i$ and $\dd_i\ri 0.$ Since $\Si_i$ are non-flat, we can assume that $|A_i|(p_i)=1. $ By Theorem \ref{theo:minimalcompactness}, a subsequence of $\td \Si_i=\Si_i-p_i$
 converges smoothly to a complete smooth minimal hypersurfaces $\Si_{\infty}$ with $|A_{\infty}|(0)=1$ and volume ratio
\beq \frac {\vol(B_r(0)\cap \Si_{\infty})}{\oo_n r^n}=1, \quad \forall \;r>0. \label{eq:G006}\eeq
(\ref{eq:G006}) implies that $\Si_{\infty}$ is a hyperplane(c.f. Corollary 1.13 of Colding-Minicozzi~\cite{[CMbook2]}), which contradicts the equality $|A_{\infty}|(0)=1$. Thus, the lemma is proved.

\end{proof}

Combining the above results, we show the following pseudolocality theorem.

\begin{theo}[\textbf{Two-sided pseudolocality}]\label{theo:pseudoA}
For any $r_0\in (0, 1], \La, T>0$,  there exist $\eta=\eta(n, \La), \ee=\ee(n, \La)>0$ satisfying
\beq
\lim_{\La\ri 0}\eta(n, \La)=\eta_0(n)>0,\quad  \lim_{\La\ri 0}\ee(n, \La)=\ee_0(n)>0\label{eq:B006}\\
\eeq
and the following properties.  Let $\{(\Si^n, \x(t)), -T\leq t\leq T\} $ be a closed smooth embedded mean curvature flow (\ref{eq:MCF}).  Assume  that
\begin{enumerate}

  \item[(1)]  the second fundamental form satisfies
$|A|(x, 0)\leq \frac 1{r_0}$ for any $ x\in C_{p_0}( B_{r_0}(p_0)\cap \Si_0) $ where $p_0=\x_0(p)$ for some $p\in \Si$;
  \item[(2)] the mean curvature of $\{(\Si^n, \x(t)), -T\leq t\leq T\}$ is  bounded by $\La$.
\end{enumerate}
Then for any $(x, t)$ satisfying
\beq
x\in C_{p_t}(\Si_t\cap B_{\frac 1{16} r_0}(p_0)), \quad t\in \Big[-\frac {\eta r_0^2}{2(\La+\La^2)}, \frac {\eta r_0^2}{2(\La+\La^2)}\Big]\cap [-T, T] \label{eq:B004}
\eeq where $p_t=\x_t(p),$  we have the estimate $$|A|(x, t)\leq \frac {1}{\ee r_0}. $$

\end{theo}

\begin{proof} The proof consists of the following steps:

\emph{Step 1.} Without loss of generality, we assume  $r_0=1$.    By Lemma \ref{lem:ratio} and the assumption (1), for any fixed $\dd>0$ there exists a constant $\td\rho_0=\td \rho_0(\dd)\in (0, \frac 12]$ such that for any $r\in (0, \td\rho_0]$ we have
\beq
\frac {\vol_{g(0)}(C_{y_0}(B_{r}(y_0)\cap \Si_0))}{\oo_n r^n}\leq 1+\dd,\quad \forall \;y_0\in C_{p_0}(B_{\frac 12}(p_0)\cap \Si_0). \label{eq:B007}
\eeq  For any $\rho_0\in (0, \td \rho_0]$, we define  \beq \eta_0(n, \rho_0, \dd):=\sup\Big\{\eta \in (0, \frac 18]\,\Big|\,
e^{\eta}\Big(1+\frac {2\eta}{\rho_0-2\eta}\Big)^n(1+\dd)\leq 1+2\dd\Big\}.\label{eq:B001}
\eeq
For any given $\La>0$, we choose $\rho_0=\rho_0(n, \dd, \La)\in (0, \td \rho_0]$  such that
\beq \rho_0(n, \dd, \La):=\sup\Big\{ \rho_0\in (0, \td \rho_0]\,\Big|\,
e^{\La (\rho_0-2\eta_0(n, \rho_0, \dd))}(1+2\dd)\leq 1+3\dd\Big\}, \label{eq:B002}
\eeq where $\eta_0$ is defined by (\ref{eq:B001}). Note that $\rho_0$ and $\eta_0$ have  positive lower bounds depending only on $n$ and $\dd$  as $\La\ri 0.$

\emph{Step 2.}
By Lemma \ref{lem:vol2}, (\ref{eq:B007}) and (\ref{eq:B001}),  for any $t\in [-T, T]$ with $\La |t|+\La^2 |t|\leq \eta_0$ we have
\beqn
\frac {\vol_{g(t)}(C_{y_t}(B_{\rho_1}(y_t)\cap \Si_t))}{\oo_n \rho_1^n}&\leq& e^{\La^2|t|}\Big(1+\frac {2\La |t|}{\rho_1}\Big)^n
\frac {\vol_{g(0)}(C_{y_0}(B_{\rho_0}(y_0)\cap \Si_0))}{\oo_n \rho_0^n}\nonumber\\
&\leq &  1+2\dd, \label{eq:H004}
\eeqn where $y_t=\x_t(\x_0^{-1}(y_0))$, $y_0$ is any point in $C_{p_0}(B_{\frac 12}(p_0)\cap \Si_0)$ and $\rho_1:=\rho_0-2\eta_0(n, \rho_0, \dd)$. Note that $\rho_1>0$ by (\ref{eq:B001}).
Let $\eta_1=\sqrt{\frac {\eta_0}{2(\La+\La^2)}}$. Then (\ref{eq:H004}) holds for any $t\in [-2\eta_1^2, 2\eta_1^2]\cap [-T, T]$.
By Lemma \ref{lem:vol3}, (\ref{eq:B002}), (\ref{eq:H004}) and the definition of $\rho_1$,  for any $s\in (0,   \rho_1]$ and any $t\in [-2\eta_1^2, 2\eta_1^2]\cap [-T, T]$ we have
\beqn
\frac {\vol_{g(t)}(C_{y_t}(B_{s}(y_t)\cap \Si_t))}{\oo_n s^n}
&\leq &e^{\La \rho_1}\frac {\vol_{g(t)}(C_{y_t}(B_{\rho_1}(y_t)\cap \Si_t))}{\oo_n \rho_1^n}\nonumber\\
&\leq& 1+3\dd.  \label{eq:G005}
\eeqn
Let $\bar \rho:=\frac 14-\eta_0$. Then by the definition of $\eta_0$ we have $\frac 18\leq \bar \rho\leq \frac 14$.   By the assumption (2), we have $p_t\in B_{\bar \rho}(p_0)\cap \Si_t\neq\emptyset$ for any $t\in [-2\eta_1^2, 2\eta_1^2]\cap [-T, T]$ since
$$|p_t-p_0|\leq \La |t|\leq 2\La \eta_1^2=\frac {\eta_0}{1+\La}< \eta_0\leq \frac 18\leq \bar \rho.$$
Using the assumption (2) again,  for any $q_t\in C_{p_t}(B_{\bar \rho}(p_0)\cap \Si_t)$ with  $t\in [-2\eta_1^2, 2\eta_1^2]\cap [-T, T]$  we have $q_0:=\x_0(\x_t^{-1}(q_t))\in C_{p_0}(B_{\frac 12}(p_0)\cap \Si_0)$ since
$$|q_0-p_0|\leq |q_0-q_t|+|q_t-p_0|< \frac 18+\bar \rho\leq\frac 12. $$
Combining this with (\ref{eq:G005}), for any $t\in [-2\eta_1^2, 2\eta_1^2]\cap [-T, T]$ and $q\in C_{p_t}(B_{\bar \rho}(p_0)\cap \Si_t)$ we have
\beq
\frac {\vol_{g(t)}(C_{q}(B_{s}(q)\cap \Si_t))}{\oo_n s^n}\leq 1+3\dd,\quad \forall\;s\in (0,  \rho_1]. \label{eq:G008}
\eeq

\emph{Step 3.}
 Suppose there exist $ \La>0$, a sequence of $\ee\ri 0, \ee\in (0, \eta_1]$ and smooth solutions to the mean curvature flow $\x_t: \Si^n\ri \RR^{n+1}$ for $t\in  [-T, T]$ with $T\geq 2\eta_1^2$ such that $|A|(x, 0)\leq 1$ for any $x\in C_{p_0}(B_1(p_0)\cap \Si_0)$, and there exists $(x_1, t_1)$ satisfying $t_1\in [-\eta_1^2, \eta_1^2]$ and $x_1\in C_{p_{t_1}}(B_{\frac 1{16}}(p_{0})\cap \Si_{t_1})$ such that
\beq Q_1:=|A|(x_1, t_1)>\frac 1{\ee}. \label{eq:I004}\eeq Note that
$p_{t_1}\in B_{\frac 1{16}}(p_{0})\cap \Si_{t_1}$ since
$$|p_{t_1}-p_0|\leq \La \eta_1^2\leq\frac {\eta_0}{2(1+\La)} \leq \frac 1{16}.$$
Fix $K>0$ such that $K>\frac 12 \La \ee$ and $2K\ee<\frac 1{16}$. Check whether there exists a point
\beq (x, t)\in  C_{x_{1, t}}(B_{KQ_1^{-1}}(x_{1})\cap \Si_t)\times [t_1-\frac 12 Q_1^{-2}, t_1] \label{eq:H008}\eeq
satisfying $|A|(x, t)>2Q_1$. Here we define $x_{1, t}=\x_t(\x_{t_1}^{-1}(x_1)).$ Note that $x_{1, t}\in B_{KQ_1^{-1}}(x_{1})\cap \Si_t$ since
$$|x_1-x_{1, t}|\leq \La|t|\leq \frac 12\La Q_1^{-2}<KQ_1^{-1},
\quad \forall\; t\in [t_1-\frac 12 Q_1^{-2}, t_1]. $$
If there is no such point, then we stop. Otherwise, we can find a point, which we denote by $(x_2, t_2)$, satisfying (\ref{eq:H008}) and $
Q_2:=|A|(x_2, t_2)>2Q_1. $ Then we check whether there exists a point
\beq (x, t)\in C_{x_{2, t}}(B_{KQ_2^{-1}}(x_{2})\cap \Si_t)\times [t_2-\frac 12 Q_2^{-2}, t_2]  \label{eq:H009}\eeq
satisfying $|A|(x, t)>2Q_2$. We can also check that $x_{2, t}\in B_{KQ_2^{-1}}(x_{2})\cap \Si_t$. If there is no such point, then we stop. Otherwise, we can find a point which we denote by $(x_3, t_3)$. Repeating the process, we can find a sequence of points $(x_k, t_k)$. Note that
$$t_k\geq t_{1}-\frac 12\Big(Q_{k-1}^{-2}+Q_{k-2}^{-2}+\cdots+Q_1^{-2}\Big)\geq t_1-\ee^2> -2\eta_1^2$$
and the Euclidean distance
\beqs
d(x_k, p_0)&\leq&d(x_k, x_{k-1 })+d(x_{k-1}, x_{k-2})+\cdots+
+d(x_{1}, p_0)\\
&\leq&K\Big(Q_{k-1}^{-1}+Q_{k-2}^{-1}+\cdots+Q_1^{-1}\Big)+d(x_1, p_0)\\
&\leq&2K\ee+\frac 1{16}\\
&\leq &\frac 18\leq \bar \rho,
\eeqs where we choose $K=\ee^{-\frac 12}$ and $\ee$ small.
Since $Q_k:=|A|(x_k, t_k)\geq 2^{k-1}Q_1\ri +\infty$ as $k\ri +\infty$, the process will stop at some finite $k$ and we get a point $(\bar x, \bar t)$ satisfying the following properties:
\begin{itemize}
  \item $(\bar x, \bar t)\in C_{p_{\bar t}}(B_{\bar \rho}(p_0)\cap \Si_{\bar t})\times (-2\eta_1^2, t_1]$;
  \item $|A|(x, t)\leq 2|A|(\bar x, \bar t)$ for any point
  $(x, t)\in C_{\bar x_{, t}}(B_{K\bar Q^{-1}}(\bar x)\cap \Si_{t})\times [\bar t-\frac 12 \bar Q^{-2}, \bar t],$ where $\bar x_{, t}:=\x_t(\x_{\bar t}^{-1}(\bar x))$ and $\bar Q:=|A|(\bar x, \bar t)$. Note that $[\bar t-\frac 12 \bar Q^{-2}, \bar t]\subset [-2\eta_1^2, t_1]$.
\end{itemize}

\emph{Step 4. } We rescale the flow by $$\td \x(p, s)=\bar Q\Big(\x(p, \bar t+\frac s{\bar Q^2})-\bar x\Big),\quad \forall\; (p, s)\in \Si\times [-(T+\bar t)\bar Q^2, (T-\bar t)\bar Q^2].$$
Then the rescaled flow $\td \Si_s:=\td \x_s(\Si)$ is a mean curvature flow satisfying  the following properties:
\begin{itemize}
  \item For any $(x, s)\in C_{\td \x_s(\td \x_0^{-1}(0))}(B_{\ee^{-\frac 12}}(0)\cap \td \Si_s)\times [-\frac 12, 0]$, we have $|A_{\td \Si_s}|(x, s)\leq 2$ and $|A_{\td \Si_0}|(0, 0)=1$;
  \item For any $r\in (0, \ee^{-\frac 12})$ we have the volume ratio
  $$\frac {\vol(C_{0}(B_r(0)\cap \td \Si_0))}{\oo_n r^n}\leq 1+3\dd.$$
  Here we used (\ref{eq:G008}),  (\ref{eq:I004}) and the facts that $K\bar Q^{-1}\leq K Q_1^{-1}\leq \ee^{\frac 12}< \rho_1$ when $\ee$ is small.
  \item The mean curvature of the flow $\td \Si_s$ satisfies $|\td H|\leq \La \bar Q^{-1}$.
\end{itemize}
Since $\ee\ri 0$ and $\bar Q\ri +\infty$, the flow  $C_{\td \x_s(\td \x_0^{-1}(0))}(B_{\ee^{-\frac 12}}(0)\cap \td \Si_s)\times [-\frac 12, 0]$ converges smoothly to a complete smooth minimal surface $\Si_{\infty}$ with $\sup_{\Si_{\infty}}|A_{\Si_{\infty}}|\leq 2$,  $|A_{\Si_{\infty}}|(0)=1$ and volume ratio
$$\frac {\vol(B_r(0)\cap \Si_{\infty})}{\oo_n r^n}\leq 1+3\dd,\quad \forall\; r>0.$$
If we choose $\dd=\frac 13\dd_0$ where $\dd_0=\dd_0(n)$ is the constant in Lemma \ref{lem:gap}, then  $\Si_{\infty}$ is a hyperplane, which contradicts $|A_{\Si_{\infty}}|(0)=1$. The theorem is proved.

\end{proof}

A direct corollary of Theorem \ref{theo:pseudoA} is the following long time pseudolocality theorem.
The long-time-pseudolocality type theorem originates from the study of the K\"ahler Ricci flow by Chen-Wang(c.f. Theorem 1.4 of Chen-Wang~\cite{[CW2]}, or Proposition 4.15 and Remark 5.3 of Chen-Wang~\cite{[CW3]}).
It will be inspiring to compare the following theorem with its K\"ahler Ricci flow counterpart.

\begin{theo}[\textbf{Long-time, two-sided pseudolocality}]\label{theo:pseudoB}
For any $r_0\in (0, 1],  T>0$,  there exist $\dd=\dd(n, r_0, T), \ee=\ee(n)>0$  with the following properties.  Let $\{(\Si^n, \x(t)), -T\leq t\leq T\} $ be a closed smooth embedded mean curvature flow (\ref{eq:MCF}).  Assume  that
\begin{enumerate}

  \item[(1)]  the second fundamental form satisfies
$|A|(x, 0)\leq \frac 1{r_0}$ for any $ x\in C_{p_0}( B_{r_0}(p_0)\cap \Si_0) $ where $p_0=\x_0(p)$ for some $p\in \Si$;
  \item[(2)] the mean curvature of $\{(\Si^n, \x(t)), -T\leq t\leq T\}$ is  bounded by $\dd$.
\end{enumerate}
Then for any $(x, t)\in C_{p_t}(\Si_t\cap B_{\frac 1{16} r_0}(p_0)) \times [-T, T]$ where $p_t=\x_t(p),$  we have the estimate $$|A|(x, t)\leq \frac {1}{\ee r_0}. $$

\end{theo}
\begin{proof} We apply Theorem \ref{theo:pseudoA} for $\La=\dd$, then we get the constant $\eta(n, \dd)$ and $\ee(n, \dd)$. By  (\ref{eq:B004}), the conclusion holds for any $t\in [-T, T]$ if
\beq \frac {\eta(n, \dd) r_0^2}{2(\dd+\dd^2)}\geq 2T.\label{eq:B005} \eeq
Since $\eta_0(n)=\lim_{\dd\ri 0}\eta(n, \dd)>0,$
 there exists a constant $\dd=\dd(n, r_0, T)$ such that (\ref{eq:B005}) holds. Note that $\lim_{\dd\ri 0}\ee(n, \dd)=\ee_0(n)>0$ by (\ref{eq:B006}). Thus, the theorem is proved.

\end{proof}

\subsection{Energy concentration property}

In \cite{[CS]}, Choi-Schoen showed the following energy concentration property for minimal surfaces. This property says that  the  energy near a point with large curvature cannot be small .
\begin{lem}\label{Choi-Schoen}(Choi-Schoen \cite{[CS]}) Fix $\rho\leq 1.$ There is a number
$\ee_0>0$ such that if  $\Si^2\subset \RR^3$ is a  minimal surface with $\p \Si
\subset\p B_{\rho}(x)$  with $|A|(x)\geq \rho^{-1}$, then
$$\int_{B_{\frac {\rho}2}(x)\cap \Si}\,|A|^2\,d\mu\geq \ee_0.$$
\end{lem}

Motivated by Choi-Schoen's result, we show that the energy concentration property holds for mean curvature flow with bounded mean curvature by using the pseudolocality theorem.

\begin{lem}[\textbf{Energy concentration}]\label{lem:energy concen1}For any $\La, K, T>0$, there exists a constant $\ee(n, \La, K, T)>0$ with the following property.
Let $\{(\Si^n, \x(t)), -T\leq t\leq T\} $ be a closed smooth embedded mean curvature flow (\ref{eq:MCF}).  Assume that $\max_{ \Si_t\times [-T, T]}|H|(p, t)\leq \La.$
Then we have
\beq
\int_{\Si_0\cap B_{Q^{-1}}(q)}\,|A|^n\,d\mu_0\geq \ee(n, \La, K, T)
\eeq whenever $q\in \Si_0$ with $Q:=|A|(q, 0)\geq K.$

\end{lem}
\begin{proof}Let  $x_0 \in \Sigma_0$ such that  $Q:=|A|(x_0, 0)\geq K$. We define the function $f(x)=|A|(x, 0) d(x, \partial \Omega)$ on $\Om:=B_{Q^{-1}}(x_0)\cap \Si_0$. Here $d$ denotes the Euclidean distance in $\RR^{n+1}$. Note that $f=0$ on the boundary $\p\Om$, if $\p\Om\neq \emptyset$,  and $f=1$ at the center point $x_0$.

\emph{Case 1.}   $\max_{\Omega} f<10$. We rescale the flow by $\td \x(p, t)=Q(\x(p, Q^{-2}t)-x_0)$ and let $\td \Om:=B_1(0)\cap \td \Si_0$, where
$\td \Si_0:=\td \x(0)(\Si)$. Note that the mean curvature of $\td \Si_t(-K^2T\leq t\leq K^2T)$ satisfies
\beq
\max_{\td \Si_t\times [-K^2T, K^2T]}|\td H|(p, t)\leq \frac {\La}{Q}\leq \frac {\La}K, \label{eq:AA001}
\eeq and $\max_{\td \Om}|\td A|d(x, \p \td \Om)=\max_{\Omega} f<10$. Thus, we have $|\td A|(x, 0) < \frac{10}{d(x, \partial \td \Omega)}$ for any $x\in B_1(0)\cap \td\Sigma_0$.  In particular, inside $B_{\frac 12}(0) \cap \td \Sigma_0$ we have  $|\td A| <20$.
Theorem \ref{theo:pseudoA} implies that there exists $\dd_0=\dd_0(n, \La, K, T)\in (0, 1)$ such that $|\td A|(p, t)\leq \frac 1{\dd_0}$ for any $t\in [-\dd^2_0, \dd_0^2]$ and $p\in \td \Si_t\cap B_{\dd_0}(0)$. Therefore, there exists $\dd_1=\dd_1(n, \La, K, T)\in (0, \dd_0)$ such that we have all higher order curvature estimates in $B_{\dd_1}(0)\cap \td \Si_t$ for  any $t\in [-\dd_1^2, \dd_1^2]$. Note that $|\td A|(0, 0)=1$, the higher order curvature estimates implies that $|\td A|(q, 0)\geq \frac 12$ on $B_{\dd_2}(0)\cap \td \Si_0$ for some $\dd_2(n, \La, K, T)\in (0, \dd_1).$ Therefore, we have
\beqs \int_{B_{Q^{-1}}(x_0)\cap \Si_0}\,|A|^n\,d\mu_0&=&
\int_{B_1(0)\cap \td \Si_0}\,|\td A|^n\,d\td \mu_0\geq \int_{B_{\dd_2}(0)\cap \td \Si_0}\,|\td A|^n\,d\td \mu_0\\&\geq& \frac 1{2^n}\vol_{\td \Si_0}(B_{\dd_2}(0)
\cap \td \Si_0)\\&\geq& \frac 1{2^n}\cdot \oo_n e^{-\frac {\La \dd_2}K}\dd_2^n,
\eeqs where we used Lemma \ref{lem:vol3} in the last inequality.

  \emph{Case 2.}  $\max_{\Omega} f \geq 10$.     Let $y_0$ be the point where $f$ achieve the maximum and  $Q':=|A|(y_0, 0)$. Note that
$$Q'\geq \frac {10}{d(y_0, \p \Om)}\geq 10Q\geq 10K.$$
We rescale the flow by $\td\x(p, t)=Q'(\x(p, Q'^{-2}t)-x_0)$ and we define
$$\td y_0:=Q'(y_0-x_0),\quad \td \Om:=B_{Q'Q^{-1}}(0)\cap \td \Si_0,$$
where $\td \Si_0:=\td \x(0)(\Si)$.  Then the function $\td f(x, 0):=|\td A|(x, 0)d(x, \p \td \Om)$ achieves the maximum at the point $\td y_0$ and the mean curvature of $\td \Si_t$ is bounded by $\frac {\La}{10K}$ for any $t\in [-Q'^2T, Q'^2T]$. Moreover, $|\td A|(\td y_0, 0)=1$ and $d(\td y_0, \p \td \Om)\geq 10.$ For any $x\in B_1(\td y_0)\cap \td \Si_0$, we have
$$|\td A|(x, 0)\leq \frac {d(\td y_0, \p \td \Om)}{d(x, \p \td \Om)}|\td A|(\td y_0, 0)\leq \frac {d(\td y_0, \p \td \Om)}{d(\td y_0, \p \td \Om)-1}\leq 2.$$
Then using the backward pseudolocality as in case $1$, we obtain that
$$\int_{B_1(\td y_0)\cap \td \Si_0}\,|\td A|^n\,d\td \mu_0\geq \ee(n, \La, K, T)$$
for some $\ee(n, \La, K, T)>0.$ Using the scaling invariance, we have
   $\int_{B_{Q'^{-1}}(y_0)\cap \Si_0} \,|A|^n \,d\mu_0 >\epsilon(n, \La, K, T)$.  By the definition of $y_0$, we have $Q'^{-1}\leq \frac 1{10}d(y_0, \p \Om)$. Therefore, $B_{Q'^{-1}}(y_0)\subset B_{Q^{-1}}(x_0)$ and we have the inequality
$$\int_{B_{Q^{-1}}(x_0)\cap \Si_0} \,|A|^n \,d\mu_0 >\epsilon(n, \La, K, T).$$ The lemma is proved.

\end{proof}

A direct corollary of Lemma \ref{lem:energy concen1} is the following result.

\begin{cor}[\textbf{$\ee$-regularity}]\label{cor:energy}
There exists $\ee_0(n)>0$ satisfying the following property.
Suppose $\{(\Si^n, \x(t)), -1\leq t\leq 1\} $ is a closed smooth embedded mean curvature flow (\ref{eq:MCF}).
Suppose that the mean curvature satisfies $\max_{ \Si_t\times [-1, 1]}|H|(p, t)\leq 1.$  For any $q\in \Si_0$, if
\begin{align}
  \int_{\Si_0\cap B_r(q)}\,|A|^n\,d\mu_0\leq \ee_0(n)   \label{eqn:PE10_5}
\end{align}
for some $r>0$, then we have
\beq \max_{B_{\frac r2}(q)\cap \Si_0}|A|\leq \max\{1, \frac 2r\}. \label{eq:AA004}\eeq
\end{cor}
\begin{proof}For any $p\in B_{\frac r2}(q)\cap \Si_0$, if $Q:=|A|(p, 0)$ satisfies $Q^{-1}<\frac r2$ and $Q> 1$, then by Lemma \ref{lem:energy concen1} we have
$$\int_{B_{r}(q)\cap \Si_0}\,|A|^n\,d\mu_0\geq \int_{B_{Q^{-1}}(p)\cap \Si_0}\,|A|^n\,d\mu_0\geq \ee_0(n),$$
where $\ee_0$ is the constant determined by choosing $K=\La=1$ in Lemma \ref{lem:energy concen1}. Therefore, we have $Q\leq 1$ or $Q^{-1}\geq \frac r2$, which implies (\ref{eq:AA004}). The corollary is proved.

\end{proof}

\subsection{Weak compactness}
In this subsection, we focus on the case $n=2.$
As in  Ricci flow \cite{[CW1]}, we study the refined sequence of mean curvature flow, which can be viewed as a sequence blown up from a rescaled mean curvature flow with bounded mean curvature and bounded energy.

\begin{defi}[\textbf{Refined sequences}]\label{defi:refined}
Let  $\{(\Si_i^2, \x_i(t)), -1\leq t\leq 1 \}$ be a one-parameter family of closed smooth embedded surfaces satisfying the mean curvature flow equation (\ref{eq:MCF}). It is called a refined sequence if the following properties are satisfied for every $i:$
\begin{enumerate}
  \item[(1)] There exists a constant $D>0$ such that $d(\Si_{i, t}, 0)\leq D$, where $d(\Si, 0)$ denotes the Euclidean distance from the point $0\in \RR^3$ to the surface $\Si\subset \RR^3.$

  \item[(2)] The mean curvature  satisfies the inequality
  \beq \lim_{i\ri +\infty}\max_{\Si_{i, t}\times [-1,  1]}|H_i|(p, t)=0. \label{eq:A203}\eeq
  \item[(3)] There is a uniform constant $\La$ such that
   \begin{align}
   \int_{\Si_{i, t}}\,|A_i|^2 \,d\mu_{i, t}\leq \La, \quad \forall \,t\in [-1, 1].    \label{eqn:PE10_2}
   \end{align}
  \item[(4)] There is uniform $N>0$ such that for all $r>0$ and $p\in \RR^3$ we have
  \begin{align}
   r^{-2}\Area_{g_i(t)}(B_r(p)\cap \Si_{i, t})\leq N, \quad \forall \,t\in [-1, 1].  \label{eqn:PE10_3}
  \end{align}
  \item[(5)] There exist  uniform constants $\bar r, \ka>0$ such that for any $r\in (0, \bar r]$ and any $p\in \Si_{i, t}$ we have
  \begin{align}
   r^{-2}\Area_{g_i(t)}(B_r(p)\cap \Si_{i, t})\geq \kappa, \quad \forall \,t\in [-1, 1].  \label{eqn:PE10_4}
  \end{align}

\end{enumerate}

\end{defi}

\begin{prop}[\textbf{Weak compactness of refined sequences}]\label{prop:EV1}
If $\{(\Si_i^2, \x_i(t)), -1\leq t\leq 1 \}$ is a refined sequence in the sense of Definition~\ref{defi:refined}, then there exists a finite set of points $\cS_0\subset \RR^3$ and a smooth embedded minimal surface $\Si_{\infty}$ such that  a subsequence of  $\{(\Si_i^2, \x_i(t)), -1<t< 1 \}$  converges in smooth topology, possibly with multiplicity at most $N_0$,  to  $\{\Si_{\infty}\}$
away from a finite set $\cS_0$ of at most $M_0$ points.  The number $N_0$ and $M_0$ can be chosen as
\begin{align}
    N_0=\left[\frac{N}{\kappa} \right]+1, \quad M_0=\left[\frac{\Lambda}{\epsilon_0} \right]+1,   \label{eqn:PE10_6}
\end{align}
where $[\cdot]$ means the integer part of a nonnegative number, $N, \kappa, \Lambda$ and $\epsilon_0$ are the numbers in (\ref{eqn:PE10_3}),  (\ref{eqn:PE10_4}), (\ref{eqn:PE10_2}) and (\ref{eqn:PE10_5}) respectively.
Furthermore, the subsequence also converges to $\Si_{\infty}$ in (extrinsic) Hausdorff distance.
\end{prop}

\begin{proof} We follow the argument of compactness of minimal surfaces (c.f. White \cite{[White3]}\cite{[White1]}, or Colding-Minicozzi \cite{[CMbook2]}).  Fix large  $\rho>0$ and let $\Om=B_\rho(0)\subset \RR^3$. By Property (1) in Definition \ref{defi:refined}, we have $\Si_{i, 0}\cap \Om\neq \emptyset$ for large $\rho$.
 For any $U\subset \Omega$, we define the measures $\nu_i$ by
$$\nu_i(U)=\int_{U\cap \Si_{i, 0}}\,|A_i|^2\,d\mu_{i, 0}\leq \La.$$
The general compactness of Radon measures implies that there is a subsequence, which we still denote by $\nu_i$, converges weakly to a Radon measure $\nu$ with
$\nu(\Om)\leq \La.$
We define the set
$$\cS_0=\{x\in \Omega\;|\;\nu(x)\geq \ee_0\},$$
where $\ee_0$ is the constant in Corollary \ref{cor:energy}.
 It follows that $\cS_0$ contains at most $\frac {\La}{\ee_0}$ points, which is independent of $\rho.$ Given any $y\in \Omega\backslash \cS_0.$ There exists some $s\in (0, \frac 15)$ such that $B_{10s}(y)\subset \Om$ and
$\nu(B_{10s}(y))<\ee_0.$ Since $\nu_i\ri \nu$, for $i$ sufficiently large we have
$$\int_{B_{10s}(y)\cap \Si_{i, 0}}\,|A_i|^2\,d\mu_{i, 0}<\ee_0. $$
Corollary \ref{cor:energy} implies that for $i$ sufficiently large we have the estimate
\beq \max_{B_{5s}(y)\cap \Si_{i, 0}}|A|(x, 0)\leq \max\{1, \frac 1{5s}\}\leq \frac 1{5s}.  \label{eq:A100}\eeq
Note that by Property (2) in Definition \ref{defi:refined} the mean curvature of $\Si_{i, t}$ tends to zero. By Theorem \ref{theo:pseudoB} there exists a universal constant $\ee>0$ such that for large $i$ and any small $s\in (0, \frac 1{5})$ we have
\beq
\max_{B_{\frac 1{16} r_0}(y)\cap \Si_{i, t}}|A|(x, t)\leq \frac 1{\ee r_0},\quad \forall\, t\in [-1, 1],
\eeq where $r_0=5s.$ Therefore, we have all higher order estimates of the second fundamental form  at any point away from the singular set.
By Theorem \ref{theo:MCFcompactness} and a diagonal sequence argument we can show that a subsequence of $\{(\Si_i^2, \x_i(t)), -1<t< 1 \}$ converges in smooth topology, possibly with multiplicities, to an embedded minimal surface $\Si_{\infty}$ away from the singular set $\cS_0.$ Property (4)-(5) imply that the multiplicity of the convergence is bounded by some constant $N_0$.
The choice of $N_0$ and $M_0$ in (\ref{eqn:PE10_6}) is clear from the above discussion.
Since $\Si_{\infty}$ is minimal and the convergence is smooth outside $\cS_0$, the same argument as in Proposition 7.14 of Colding-Minicozzi \cite{[CMbook2]} shows that $\Si_{\infty}\cup \cS_0$ is a smooth embedded minimal surface and the convergence is also in Hausdorff distance. The proposition is proved.

\end{proof}

To study the multiplicity, we define a function
\beq
\Te(x, r, t) :=\lim_{i\ri +\infty}\frac {\Area_{g_i(t)}(\Si_{i, t}\cap B_r(x))}{\pi r^2},\quad
\forall \;(x, t)\in \Si_{\infty}\times (-1, 1).\label{eqD:002}
\eeq
Then the multiplicity at $(x, t)\in \Si_{\infty}\times (-1, 1)$  is give by
\beq
\frak m(x, t) :=\lim_{r\ri 0}\Te(x, r, t).  \label{eqn:PE18_2}
\eeq
It is clear that $\frak m(x, t)$ is an integer.

\begin{lem}
Under the  assumption of Proposition~\ref{prop:EV1}, the function $\frak m(x, t)$ is a constant integer on $\Si_{\infty}\times (-1, 1)$. Namely, $\frak m(x, t)$ is independent of $x$ and $t$.
\label{lma:PE16_1}
\end{lem}

\begin{proof}We divide the proof into several steps.\\

{\it Step 1. For each $t\in (-1, 1)$, $\frak m(x, t)$ is  constant on $\Si_{\infty}\b \cS_0$.} Fix $t_0\in (-1, 1)$ and  $x_0\in \Si_{\infty}\b \cS_0$.
 Since $x_0$ is a regular point, there exists $r_0>0$ such that for large $i$,
\beq
|A|(x, t_0)\leq \frac 1{r_0}, \quad \, \forall\; x\in B_{r_0}(x_0)\cap \Si_{i, t_0}.\label{eqE:001}
\eeq
By Lemma \ref{lem:graph1}, we can assume $r_0$ small such that $B_{r_0}(x_0)\cap \Si_{\infty}$ can be written as a graph over the tangent plane of $\Si_{\infty}$ at $x_0$.
Let $r_1=\frac {r_0}{4}$. For any $p\in B_{r_1}(x_0)\cap \Si_{i, t_0}$, we have $B_{\frac {r_0}2}(p)\subset B_{r_0}(x_0)$. Thus, (\ref{eqE:001}) implies that
\beq
|A|(x, t_0)\leq \frac 2{r_0}, \quad \, \forall\; x\in B_{\frac {r_0}2}(p)\cap \Si_{i, t_0}.\label{eqE:002}
\eeq By Lemma \ref{lem:ratio} for any $\dd>0$ there exists $\rho_0=\rho_0(r_0, \dd)>0$ such that for any $r\in (0, \rho_0)$ and any $p\in B_{r_1}(x_0)\cap \Si_{i, t_0}$  we have
\beq
 \frac {\Area_{g_i(t_0)}(C_p(B_r(p)\cap \Si_{i, t_0}))}{\pi r^2}\leq 1+\dd. \label{eqE:003}
\eeq
On the other hand, by Lemma~\ref{lem:vol3} we can choose $\rho_0$ small such that for any $r\in (0, \rho_0)$,   on each component of $B_r(p)\cap \Si_{i, t_0}$ we obtain
\begin{align}
  \frac {\Area_{g_i(t_0)}(C_p(B_r(p)\cap \Si_{i, t_0}))}{\pi r^2}\geq 1-\dd.  \label{eqn:PE18_1}
\end{align}
Suppose that $B_{r_1}(x_0)\cap \Si_{i, t_0}$ has $m_i$ connected components, where $m_i$ is an integer bounded by a constant independent of $i$ by Proposition \ref{prop:EV1}. After taking a subsequence of $\{\Si_{i, t_0}\}$ if necessary, we can assume that $m_i$ are the same
integer denoted by $m$ with $m\geq 1.$
For any $x\in B_{\frac {r_1}2}(x_0)\cap \Si_{\infty}$, we denote by $\al_x$ the normal line passing through $x$ of $\Si_{\infty}$. Since each component of $B_{r_1}(x_0)\cap \Si_{i, t_0}$ converges to $B_{r_1}(x_0)\cap \Si_{\infty}$ smoothly and $B_{r_1}(x_0)\cap \Si_{\infty}$ is a graph over the tangent plane of $\Si_{\infty}$ at $x_0$, $\al_x$ intersects transversally each component of $\Sigma_{i, t_0}$  at exactly one point.
Suppose that
$$\al_x\cap \Big(B_{r_1}(x_0)\cap \Si_{i, t_0}\Big)=\{p_i^{(1)}, p_i^{(2)}, \cdots, p_i^{(m)}\}. $$
Then (\ref{eqE:003}) and (\ref{eqn:PE18_1}) imply that for any integer $j$ with $1\leq j\leq m$ and any $r\in (0, \rho_0)$,
\beq
1-\delta \leq \frac {\Area_{g_i(t_0)}(C_{p_i^{(j)}}(B_r(p_i^{(j)})\cap \Si_{i, t_0}))}{\pi r^2}\leq 1+\dd. \label{eqE:004}
\eeq
Since for any $1\leq j\leq m$ we have $p_i^{(j)}\ri x$ and
$C_{p_i^{(j)}}(B_r(p_i^{(j)})\cap \Si_{i, t_0})$ converges smoothly to $B_r(x)\cap \Si_{\infty}$,  (\ref{eqE:004}) implies that
 \beq
 m(1-\delta) \leq \lim_{i\ri +\infty}\frac {\Area_{g_i(t_0)}(B_r(x)\cap \Si_{i, t_0})}{\pi r^2}\leq m(1+\dd).
 \eeq
In other words, for any $x\in B_{\frac {r_1}2}(x_0)\cap \Si_{\infty}$ and any $r\in (0, \rho_0)$ we have
\beq m(1-\delta) \leq \Te(x, r, t_0)\leq m(1+\dd). \label{eqE:005}\eeq
Taking $r\ri 0$ in (\ref{eqE:005}), we have
$$\m(x, t_0)=m,\quad \forall\; x\in B_{\frac {r_1}2}(x_0)\cap \Si_{\infty}.$$
By the connectedness of $\Si_{\infty}\b \cS_0$, we know that $\m(x, t_0)$ is constant on  $\Si_{\infty}\b \cS_0$.\\

{\it Step 2. For each $t\in (-1, 1)$, $\frak m(x, t)$ is  constant on $\Si_{\infty}$.} It suffices to consider a singular point $p_0\in \cS_0$. Fix $t_0\in (-1, 1)$. Suppose that $B_{r}(p_0)\cap \Si_{\infty}$ has no other singular points except $p_0$ for any $r\in (0, r_0).$ Then all points in $(B_r(p_0)\b B_{\ee}(p_0))\cap \Si_{\infty}$ are regular and $(B_r(p_0)\b B_{\ee}(p_0))\cap \Si_{i, t_0}$ has $m$ connected components.  Thus, we have
\beqn
\Area_{g_i(t_0)}(\Si_{i, t_0}\cap B_r(p_0))&\leq&\Area_{g_i(t_0)}(\Si_{i, t_0}\cap (B_r(p_0)\b B_{\ee}(p_0)))+\Area_{g_i(t_0)}(\Si_{i, t_0}\cap B_{\ee}(p_0))\nonumber\\
&\leq &\Area_{g_i(t_0)}(\Si_{i, t_0}\cap (B_r(p_0)\b B_{\ee}(p_0)))+N \ee^2,\label{eqF:001}
\eeqn where we used (\ref{eqn:PE10_3}) in the last inequality. Since each  component of $\Si_{i, t_0}\cap (B_r(p_0)\b B_{\ee}(p_0)) $ converges to $(B_r(p_0)\b B_{\ee}(p_0))\cap \Si_{\infty}$ smoothly, we have
\beq
\lim_{i\ri +\infty}\Area_{g_i(t_0)}(\Si_{i, t_0}\cap (B_r(p_0)\b B_{\ee}(p_0)))=m\,\Area_{g_\infty}(\Si_{\infty}\cap (B_r(p_0)\b B_{\ee}(p_0))),\label{eqF:002}
\eeq where $m$ is the number of components of $\Si_{i, t_0}\cap (B_r(p_0)\b B_{\ee}(p_0))$. Note that  $m$ is also the multiplicity at each regular point in $\Si_{\infty}$ by Step 1.
Combining (\ref{eqF:001}) with (\ref{eqF:002}), we have
\beqn &&m\,\Area_{g_\infty}(\Si_{\infty}\cap (B_r(p_0)\b B_{\ee}(p_0)))\nonumber\\&\leq&
\lim_{i\ri +\infty}\Area_{g_i(t_0)}(\Si_{i, t_0}\cap B_r(p_0))\nonumber\\&\leq& m\,\Area_{g_\infty}(\Si_{\infty}\cap (B_r(p_0)\b B_{\ee}(p_0)))+N \ee^2.\label{eqF:003}
\eeqn
Taking $\ee\ri 0$ in (\ref{eqF:003}), we have
\beq
\lim_{i\ri +\infty}\Area_{g_i(t_0)}(\Si_{i, t_0}\cap B_r(p_0))=m\,\Area_{g_\infty}(\Si_{\infty}\cap B_r(p_0)). \label{eqF:004}
\eeq Thus, we have
\beqs
\m(p_0, t_0)&=&\lim_{r\ri 0}\frac {\Area_{g_i(t_0)}(\Si_{i, t_0}\cap B_r(p_0))}{\pi r^2}\\
&=&m \lim_{r\ri 0}\frac {\Area_{g_\infty}(\Si_{\infty}\cap B_r(p_0))}{\pi r^2}=m.
\eeqs
This implies that the multiplicity of each singular point is the same as that of any regular point. \\

{\it Step 3. $\frak m(x, t)$ is constant in $t$.}
Let $x_0\in \Si_{\infty}$ and $\ee_i:=\max_{\Si_{i, t}}|H|\ri 0$.
Similar to the proof of (\ref{eqG:001}), for any $t_1, t_2\in (-1, 1)$ we have
\beq
|\x_i(p, t_2)-x_0|\leq |\x_i(p, t_1)-x_0|+\ee_i|t_1-t_2|.
\eeq
Thus, we have
\beq \x(t_1)^{-1}(B_{r}(x_0)\cap \Si_{i, t_1})\subset \x(t_2)^{-1}(B_{r_1}(x_0)\cap \Si_{i, t_2}),  \label{eqG:003}\eeq
where
\beq
r_1=r+\ee_i |t_2-t_1|. \label{eqG:002}
\eeq
Recall that the evolution of area element along mean curvature flow is dominated by $|H|^2$ as in (\ref{eqn:PE24_1}).
Therefore (\ref{eqG:003}) implies that
\begin{align}
\Area_{g_i(t_1)}(B_{r}(x_0)\cap \Si_{i, t_1} )&\leq  e^{\epsilon_i^2|t_2-t_1|}  \Area_{g_i(t_2)}(B_{r_1}(x_0)\cap \Si_{i, t_2}).  \label{eqG:004}
\end{align}
Combining (\ref{eqG:004}) with (\ref{eqG:002}), we have
\beqn
\Te(x_0, r,  t_1)&=&\lim_{i\ri +\infty}\frac {\Area_{g_i(t_1)}(B_{r}(x_0)\cap \Si_{i, t_1} )}{\pi r^2}\nonumber\\
 &\leq& \lim_{i\ri +\infty} e^{\epsilon_i^2|t_2-t_1|}  \frac {\Area_{g_i(t_2)}(B_{r_1}(x_0)\cap \Si_{i, t_2})}{\pi r^2} \nonumber\\
&=& \lim_{i\ri +\infty}\frac {\Area_{g_i(t_2)}(B_{r_1}(x_0)\cap \Si_{i, t_2})}{\pi r^2}\nonumber\\
&=&\Te(x_0, r, t_2).  \label{eqG:005}
\eeqn
Letting $r\ri 0$ in (\ref{eqG:005}), we have
$$\frak m(x_0, t_1)\leq \frak m(x_0, t_2).$$
Since $t_1$ and $ t_2$ are arbitrary in $(-1, 1)$,  we have that $\frak m(x, t)$ is constant in $t$.

\end{proof}

\section{Multiplicity-one convergence of the rescaled mean curvature flow}
In this section, we show that a rescaled mean curvature flow with mean curvature exponential decay will converge smoothly to a plane with multiplicity one.

\begin{theo}
\label{theo:removable}
Let $\{(\Si^2, \x(t)), 0\leq t<+\infty\}$ be a rescaled mean curvature flow
\beq \Big(\pd {\x}t\Big)^{\perp}=-\Big(H-\frac 12\langle \x, \n\rangle\Big)\n\label{eq:BB001}\eeq
 satisfying
\beq d(\Si_t, 0)\leq D,\quad \hbox{and}\quad \max_{\Si_t}|H(p, t)|\leq \La_0 e^{-\frac t2} \label{eq:BB002}\eeq
for two constants $D, \La_0>0$.  Then there exists a sequence of times $t_j\ri +\infty$ such that $\Si_{t_j}$  converge in smooth topology to a  plane passing through the origin with multiplicity one.
\end{theo}

We sketch the proof of Theorem \ref{theo:removable}.  First, we show the weak compactness for any sequence of the rescaled mean curvature flow in Lemma \ref{lem:A001}. Suppose that the multiplicity is at least two.
By using the decomposition of spaces(c.f. Definition~\ref{def:GD18_1}) we can select a special sequence $\{t_i\}$ in Lemma \ref{lma:GC05_3}. This special sequence is needed to control the upper bound of the function $w_i$
by using the parabolic Harnack inequality (c.f. Lemma \ref{lem:estimates}).  Then we can take the limit for the function $w_i$ and obtain a positive function $w$ with uniform  bounds(c.f. Proposition~\ref{prn:PE14_1}).
The function $w$ satisfies the linearized rescaled mean curvature flow equation.  The bounds of $w$ imply  the $L$-stability of the limit plane (c.f. Lemma~\ref{lma:GB24_2} and Lemma \ref{lem:Z}).
However, the plane is not $L$-stable and we obtain a contradiction.

\subsection{Convergence away from singularities}
\label{subsec:conv}

\begin{lem}\label{lem:A001}
Under the assumption of Theorem \ref{theo:removable}, for any sequence $t_i\ri +\infty$, there is a plane $\Si_{\infty}$ passing through the origin and a finite set $\cS_0\subset \Si_{\infty}$ of  points satisfying the following properties. For any $T>0$,
 there is a subsequence, still denoted by $\{t_i\}$,  such that $\{\Si_{t_i+t}, -T<t<T\}$ converges in smooth topology,
possibly with multiplicities at most $N_0$, to the plane $\Si_{\infty}$  away from the space-time singular set
$\cS=\{(x, t)\,|\,t\in (-T, T), x\in e^{\frac t2}\cS_0 \}$.

\end{lem}

\begin{proof}The proof divides into the following steps.

\textit{Step 1. The energy of $\Si_t$ is uniformly bounded along the flow (\ref{eq:BB001}).} In fact, we rescale the flow $\Si_t$ by
\beq
s=1-e^{-t},\quad \hat \Si_s=\sqrt{1-s}\,\Si_{-\log(1-s)} \label{eq:MCF2}
\eeq such that $\{\hat \Si_s, 0\leq s<1 \}$   is a mean curvature flow satisfying (\ref{eq:MCF}).  Moreover, the mean curvature $\hat H$ of $\hat \Si_s$ satisfies
\beq
\max_{\hat \Si_s}|\hat H|=\frac 1{\sqrt{1-s}}\max_{\Si_t}|H|=e^{\frac t2}\max_{\Si_t}|H|\leq \La_0,
\eeq where we used the assumption (\ref{eq:BB002}).
Note that the scalar curvature $\hat S$ of $\hat \Si_s$ satisfies  $\hat S=\hat H^2-|\hat A|^2$, where   $\hat A$ denotes   the second fundamental form of $\hat \Si_s$.
By the Gauss-Bonnet theorem we have
\begin{align}
\int_{\hat \Si_s}\,|\hat A|^2\,d\hat \mu_s=\int_{\hat \Si_s}\,|\hat H|^2\,d\hat \mu_s-4\pi\chi(\Si_0)\leq \La_0^2 \Area(\Si_0)-4\pi\chi(\Si_0) =: \Lambda.  \label{eqn:PE10_1}
\end{align}
where we used the fact that $\Area(\hat \Si_s)$ is non-increasing  in $s$. Here $\chi(\Si)$ denotes the  Euler characteristic of $\Si$.  Therefore, by Lemma \ref{lem:invariant} the energy of $\Si_t$ satisfies the inequality
\beq \int_{\Si_t}\,|A|^2\,d\mu_t\leq \Lambda. \label{eq:AA005} \eeq

\textit{Step 2. For any sequence $t_i \ri + \infty$, we can obtain a refined sequence converging to a limit minimal surface $\tilde{\Sigma}_{\infty}$.}
 For any sequence $t_i\ri +\infty,$ we can rescale the flow $\Si_t$ by
\beq
s=1-e^{-(t-t_i)},\quad \td \Si_{i, s}=\sqrt{1-s}\;\Si_{ t_i-\log(1-s)} \label{eq:BB003}
\eeq  such that for each $i$ the flow   $\{\td \Si_{i, s}, 1-e^{t_i}\leq s< 1\}$ is a mean curvature flow satisfying (\ref{eq:MCF}) with the following properties:
\begin{enumerate}
  \item[$(a)$.]  For any small $\la>0$, the mean curvature of $\td \Si_{i, s}$ satisfies
  $$\lim_{i\ri +\infty}\max_{\td \Si_{i, s}\times [1-e^{t_i}, 1-\la]}|\td H_i|(p, s)= 0;$$
\item[$(b)$.] The energy of $\td \Si_{i, s}$ satisfies (\ref{eq:AA005});
  \item[$(c)$.] Uniform upper bound on the area ratio;
  \item[$(d)$.]  Uniform lower bound on the area ratio;
  \item[$(e)$.] There exists a constant $D'>0$ such that $d(\td \Si_{i, s}, 0)\leq D'$ for any $i$.
\end{enumerate}
In fact, Property $(a)$ and $(e)$ follow from the assumption (\ref{eq:BB002}), and Property $(b)$ follows from (\ref{eq:AA005}).
Property $(c)$  follows  from Lemma \ref{lem:vol} and Lemma \ref{lem:invariant},  and Property $(d)$ follows directly from Lemma \ref{lem:vol3}. Therefore, by Definition \ref{defi:refined} for any $T_0>2$, small $\la\in (0, 1)$ and any $s_0\in [-T_0+1, -\la]$ the sequence $\{\td \Si_{i, s_0+\tau}, -1<\tau<1\}$ is a refined sequence. By Proposition \ref{prop:EV1} a subsequence of  $\{\td \Si_{i, s_0+\tau}, -1<\tau<1\}$  converges in smooth topology, possibly with multiplicity at most $N_0$, to a smooth embedded minimal surface $\td \Si_{\infty}$ away from a finite set of points $\td \cS=\{q_1, \cdots, q_l\}$ such that
\begin{align}
 l \leq \left[ \frac{\Lambda}{\epsilon_0}\right] +1=M_0,   \label{eqn:PE10_7}
\end{align}
where we used  (\ref{eqn:PE10_6}) and (\ref{eq:AA005}).

\begin{claim}\label{claim3}
Under the above assumptions, there exists a subsequence of $\{\td \Si_{i, s}, -T_0<s<1-\la\}$ such that it converges  in smooth topology, possibly with multiplicity at most $N_0$, to the limit minimal surface $\td \Si_{\infty}$ away from the singular set $\td \cS=\{q_1, \cdots, q_l \}$. Furthermore, $\td \Si_{\infty}$ and $\td \cS$ are independent of $\;T_0$ and $\la. $
\end{claim}

\begin{proof}

Let $T_0=N+\bb$ where $N\in \NN, N\geq 2$ and $\bb\in [0, 1)$. For the interval $I_0:=(-2, 0)$, we have a subsequence of $\{t_i\}$, which we denote by $\{i_k^{(1)}\}$, such that  $\{\td \Si_{i_k^{(1)}, s}, s\in I_0\}$  converges  in smooth topology, possibly with multiplicity at most $N_0$, to a limit minimal surface $\td \Si_{\infty}$ away from a singular set $\td \cS=\{q_1, \cdots, q_l \}$. Consider the interval $I_0':=(-1-\la, 1-\la)$ with $\la\in (0, 1)$. Since $I_0'\cap I_0\neq\emptyset$, we can take a further subsequence of $\{i_k^{(1)}\}$, denoted by $\{i_k^{(2)}\}$,  such that  $\{\td \Si_{i_k^{(2)}, s}, s\in I_0'\}$ converges to the same limit surface $\td \Si_{\infty}$ away from the same singular set $\td \cS$. Similarly, we consider $I_2=(-2-\bb, -\bb)$. Since $I_0\cap I_2\neq \emptyset$, we can take a subsequence of $\{i_k^{(2)}\}$, denoted by $\{i_k^{(3)}\}$,  such that  $\{\td \Si_{i_k^{(3)}, s}, s\in I_2\}$ converges to the same limit surface $\td \Si_{\infty}$ away from the same singular set $\td \cS$. Repeating this process for the interval $I_j=(-j-\bb, -j-\bb+2)$ for $j= 3, \cdots, N$  and we get a subsequence $\{i_k^{(N+1)}\}$ of $\{t_i\}$, such that  $\{\td \Si_{i_k^{(N+1)}, s}, -T_0<s<1-\la\}$  converges to   $\td \Si_{\infty}$ away from   $\td \cS$. By the construction of $\td \Si_{\infty}$ and $\td \cS$, we know that $\td \Si_{\infty}$ and $\td \cS$ are independent of the choice of $T_0$ and $\la$. The Claim is proved.

\end{proof}

\textit{Step 3. Each limit $\td\Si_{\infty}$ must be a plane through the origin.}
In fact, by Huisken's monotonicity formula, along the rescaled mean curvature flow (\ref{eq:BB001}) we have
\beq
\frac d{dt}\int_{\Si_t}\,e^{-\frac {|\x|^2}4}\,d\mu_t=-\int_{\Si_t}\,e^{-\frac {|\x|^2}4}\Big|H-\frac 12\langle\x, \n\rangle\Big|^2\,d\mu_t.
\eeq This implies that
$$\int_0^{\infty}\,\int_{\Si_t}\,e^{-\frac {|\x|^2}4}\Big|H-\frac 12\langle\x, \n\rangle\Big|^2\,d\mu_t<+\infty.$$
For any $t_i\ri +\infty$, we rescale the flow (\ref{eq:BB001}) by (\ref{eq:BB003}) such that for each $i$ the flow   $\{\td \Si_{i, s}, 1-e^{t_i}\leq s< 1\}$ is a mean curvature flow satisfying (\ref{eq:MCF}) and we denote the solution by $\td \x_{i, s}$. Therefore, for fixed $T_0>0$, small $\la>0$ and large $i$ we have
\beqs &&
\lim_{i\ri +\infty}\,\int_{-T_0}^{1-\la}\,ds\int_{\td \Si_{i, s}}\,
e^{-\frac{|\td \x_{i, s}|^2}{4(1-s)}}\,\Big|\td H_i-\frac{1}{2(1-s)}\langle \td \x_{i, s}, \n\rangle\Big|^2\,d\td \mu_{i, s}\\&=&
\lim_{t_i\ri +\infty}\int_{t_i-\log (1+T_0)}^{t_i-\log \la}\,\int_{\Si_t}\,e^{-\frac {|\x|^2}4}\Big|H-\frac 12\langle\x, \n\rangle\Big|^2\,d\mu_t=0.
\eeqs
Since $\{\td \Si_{i, s}, -T_0< s< 1-\la\}$ converges locally smoothly, possibly with multiplicity at most $N_0$,  to $\td \Si_{\infty}$ away from $\td \cS$, we have
$$\int_{-T_0}^{1-\la}\,ds\int_{\td \Si_{\infty}}\,
e^{-\frac {|\td \x_{\infty}|}{4(1-s)}}\Big|\td H-\frac 1{2(1-s)}\langle\td \x_{\infty}, \n\rangle\Big|^2\,d\td \mu_{\infty, s}=0.$$
Therefore, $\{(\td \Si_{\infty}, \td \x_{\infty}(p, s)), -T_0<s<1-\la\}$ satisfies the equation
$$\td H-\frac 1{2(1-s)}\langle\td \x_{\infty}, \n\rangle=0$$
away from the singular set $\td \cS.$ Since $\td \Si_{\infty}$ is  a smooth embedded minimal surface, we have $\td H=\langle\td \x_{\infty}, \n\rangle=0$. By Lemma \ref{lem:selfshrinker}  we know $\td \Si_{\infty}$ must be a plane passing through the origin.
 Let $\x_i(p, t)=\x(p, t_i+t)$ and $\Si_{i, t}=\Si_{t_i+t}$.
Since $\{\td \Si_{i, s}, -T_0<s<1-\la\}$ converges locally smoothly to $\td \Si_{\infty}$ away from $\td \cS$, the flow $\{\Si_{i, t}, -\log(1+T_0)<t<-\log \la \}$ also converges locally smoothly to the plane $\td \Si_{\infty}$ away from the singular set
$\cS=\{(x, t)\;|\;t\in (-\log(1+T_0), -\log \la),\;x\in e^{\frac t2}\td \cS\}.$ Moreover, by Claim \ref{claim3} $\td \Si_{\infty}$ and $\td \cS$ are independent of $T_0$ and $\la$.
The lemma is proved.
\end{proof}

 \subsection{Decomposition of spaces and a ``monotone  decreasing" quantity}
 \label{subsec:decompose}

 Note that if one limit plane  in Lemma~\ref{lem:A001} has multiplicity one, then the proof of Theorem~\ref{theo:removable} is done.
 Therefore, we can assume that every limit plane has multiplicity more than one.
 Consequently, we can decompose the space as follows.

 \begin{figure}
 \begin{center}
 \psfrag{A}[c][c]{Thick part $\mathbf{TK}$}
 \psfrag{B}[c][c]{\color{blue} Thin part $\mathbf{TN}$ }
 \psfrag{C}[c][c]{\color{red} High curvature neighborhood $\mathbf{H}$}
 \includegraphics[width=0.5 \columnwidth]{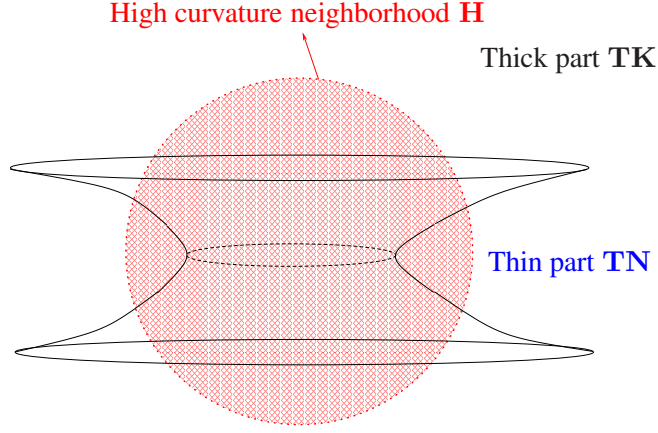}
 \caption{Decomposition of the space}
 \label{fig:decomposition}
 \end{center}
 \end{figure}

 \begin{defi} \label{def:GD18_1}
 \begin{enumerate}
                                   \item[(1).] We define the set $\mathbf{S}=\mathbf{S}(\epsilon, \Sigma_t)=\{y \in \Sigma_t\; |\; |y|<\ee^{-1},\; |A|(y, t)>\epsilon^{-1}  \}.$
                                   \item[(2).]  The ball $B_{\ee^{-1}}(0)$ can be decomposed into three parts as follows:
  \begin{itemize}
  \item the high curvature part $\mathbf{H}$, which is defined by  $\mathbf{H}=\mathbf{H}(\epsilon, \Sigma_t)=\{x\in \RR^3\; |\; |x|<\epsilon^{-1},  d(x,\mathbf{S}) <\frac {\epsilon}2 \}.$ Here $d$ denotes the Euclidean distance in $\RR^3$.
        \item  the thick part $\mathbf{TK}$, which is defined by
       \begin{align*}
          \mathbf{TK}&=\mathbf{TK}(\epsilon, \Sigma_t)\\&=\Big\{x\in \RR^3\; \Big| \;|x|<\epsilon^{-1}, \; \textrm{there is a continuous curve $\gamma \subset B_{\epsilon^{-1}}(0) \backslash (\mathbf{H} \cup \Sigma_t)$} \\
                             &\qquad \qquad \textrm{ connecting $x$ and some $y$ with} \; B(y,\epsilon) \subset B_{\epsilon^{-1}}(0) \backslash (\mathbf{H} \cup \Sigma_t)\Big\}.
       \end{align*}
  \item the thin part $\mathbf{TN}$, which is defined by  $\mathbf{TN}=\mathbf{TN}(\epsilon, \Sigma_t)=B_{\epsilon^{-1}}(0) \backslash (\mathbf{H} \cup \mathbf{TK})$.

  \end{itemize}
                                 \end{enumerate}
 \end{defi}

We remind the readers that the decomposition in the above definition depends on the fact that the limit plane has multiplicity more than one.
 Intuitively, the high curvature part  $\mathbf{H}$ is the neighborhood of points with large second fundamental form(c.f. Figure~\ref{fig:decomposition}).
 The thin part $\mathbf{TN}$ is the domain between the top and bottom sheets.
 The thick part is the union of path connected components of  the domain ``outside" the sheets.
 Note that in the above definition, the existence of curve $\gamma$ is to guarantee the path-connectedness.  Because of the boundary issue,  some points in $\mathbf{TK}$ may be very close
 to the points in $\mathbf{TN}$ or $\mathbf{H}$.

 \begin{rem}
  The decomposition of space is motivated by the decomposition of Ricci flow time slices in Chen-Wang~\cite{[CW2]}.
 \label{rem:MD05_1}
 \end{rem}

 \begin{lem} \label{lma:GD18_1}
   For each fixed $\epsilon$, we have
   $
       \displaystyle \lim_{t \to \infty}   |\mathbf{TN}(\epsilon, \Sigma_{t})| =0.
   $ Here the notation $|\Om|$ denotes the volume of $\Om$ with respect to the standard metric on $\RR^3.$
 \end{lem}

\begin{proof}
  For otherwise, we can find an $\epsilon$ and a sequence $t_i \to \infty$ such that
\begin{align}
  \lim_{i \to \infty}  |\mathbf{TN}(\epsilon, \Si_{t_i})| \geq \kappa_0>0.
\label{eqn:GD18_4}
\end{align}
However, $\Sigma_{t_i}$ converges locally smoothly to a plane $\Sigma_{\infty}$ passing through the origin away from the singular set $\cS_0\subset \Si_{\infty}.$ Therefore, the sets $\mathbf{S}(\epsilon, \Si_{t_i})$ converge to  $\cS_0$ as $i\ri +\infty$.
Without loss of generality, we can assume that $\Sigma_{\infty}$ is determined by $x_3=0$ and the singular set $\cS_0=\{(1,0,0)\}$.
Then the high curvature parts $\mathbf{H}(\epsilon, \Si_{t_i})$ converge to $B_{\frac {\ee}2}((1,0,0))$(c.f. Figure~\ref{fig:limitdecomposition}).

 \begin{figure}
 \begin{center}
 \psfrag{A}[c][c]{$\mathbf{TK}=B_{\epsilon^{-1}}(0) \backslash \left( \Sigma_{\infty} \cup B_{\frac{\epsilon}{2}}(\mathcal{S}_0)\right)$}
 \psfrag{B}[c][c]{\color{blue} $\mathbf{TN}=\Sigma_{\infty} \cap \left( B_{\epsilon^{-1}}(0) \backslash B_{\frac{\epsilon}{2}}(\mathcal{S}_0)\right)$}
 \psfrag{C}[c][c]{\color{red}  $\mathbf{H}=B_{\frac{\epsilon}{2}}(\mathcal{S}_0)$}
 \psfrag{D}[c][c]{$\mathcal{S}_0$}
 \psfrag{E}[c][c]{$0$}
 \psfrag{x}[c][c]{$x_1$}
 \psfrag{y}[c][c]{$x_2$}
 \psfrag{z}[c][c]{$x_3$}
 \includegraphics[width=0.5 \columnwidth]{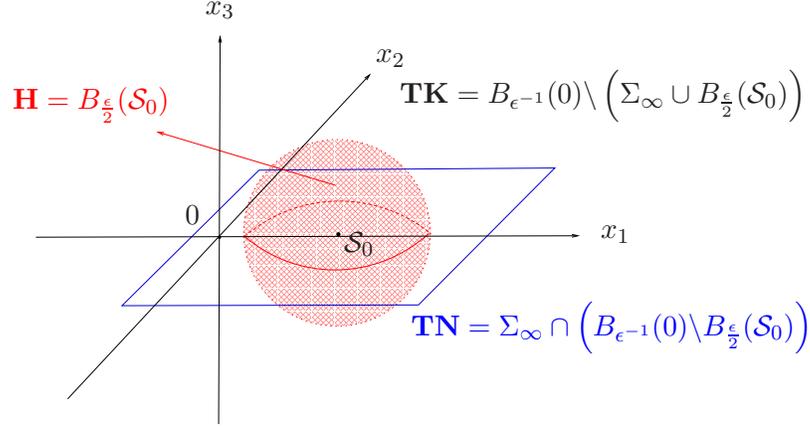}
 \caption{Decomposition in the limit space}
 \label{fig:limitdecomposition}
 \end{center}
 \end{figure}

Note that $B_{\ee^{-1}}(0) \backslash (\Sigma_{\infty} \cup   B_{\frac {\ee}2}((1,0,0)))$ contains two parts $\hat{B}^{+} \cup \hat{B}^{-}$, where
\begin{align*}
  &\hat{B}^{+}=\Big\{(x_1, x_2, x_3) \;\Big|\; x_1^2+x_2^2+x_3^2<\epsilon^{-2}, \; (x_1-1)^2+x_2^2+x_3^2>\frac {\epsilon^2}4, \; x_3>0\Big\}, \\
  &\hat{B}^{-}=\Big\{(x_1, x_2, x_3)  \;\Big|\; x_1^2+x_2^2+x_3^2<\epsilon^{-2}, \; (x_1-1)^2+x_2^2+x_3^2>\frac {\epsilon^2}4, \;x_3<0\Big\}.
\end{align*}
Fix an arbitrary point $x \in \hat{B}^{+}$. It is not hard to see that $x$ can be connected to $(0,0,1)$ by a continuous curve $\gamma \subset \hat{B}^{+}$.
Furthermore, it is clear that
$B_{\epsilon}((0,0,1)) \cap B_{\epsilon}((1,0,0))=\emptyset$.
Similar argument applies for $\hat{B}^{-}$. Therefore, we see that $\hat{B}^{+} \cup \hat{B}^{-}$ are contained in the limit of $\mathbf{TK}(\epsilon, \Si_{t_i})$.
Recall that $\mathbf{H}(\epsilon, \Si_{t_i})$ converges to $B_{\frac {\ee}2}((1,0,0))$.  Following  the  definition of $\mathbf{TN}$,  the Hausdorff limit of
$\mathbf{TN}(\epsilon, \Si_{t_i})$ is a subset of $( B_{\ee^{-1}}(0) \cap \Sigma_{\infty} )\backslash B_{\frac {\ee}2}((1,0,0))$.
Thus, for large $t_i$ we know $\mathbf{TN}(\epsilon, \Si_{t_i})$ lies in a small neighborhood of $( B_{\ee^{-1}}(0) \cap \Sigma_{\infty} )\backslash B_{\frac {\ee}2}((1,0,0))$. In other words, for any $\dd>0$ and large $t_i$ we have $\mathbf{TN}(\epsilon, \Si_{t_i})\subset \Om_{\ee, \dd}$, where
$$\Om_{\ee, \dd}=\Big\{(x_1, x_2, x_3)\in B_{\ee^{-1}}(0)\;\Big|\;(x_1-1)^2+x_2^2+x_3^2\geq \frac {\ee^2}{16}, \; |x_3|<\dd\Big\}.$$
Then for large $t_i$,
\beq |\mathbf{TN}(\epsilon, \Si_{t_i})|\leq |\Om_{\ee, \dd}|. \label{eqn:001} \eeq
Taking  $i\ri \infty$ in (\ref{eqn:001}) we have that the limit
of $|\mathbf{TN}(\epsilon, \Si_{t_i})|$ is  arbitrary small.
This contradicts the assumption (\ref{eqn:GD18_4}).
\end{proof}

It is not hard to observe that
\begin{lem}\label{lma:GD18_3}
 If $\Sigma_{\infty}$ has multiplicity more than one, then for sufficiently large $t_i$ we have $$|\mathbf{TN}(\epsilon, \Si_{t_i})|>0.$$
\end{lem}

\begin{proof}Since $\Si_t$ is embedded and $\Si_{t_i}$ converges locally smoothly to the limit plane $\Si_{\infty}$, all components of $(\Si_{t_i}\cap B_{\ee^{-1}}(0))\b \mathbf{H}(\ee, \Si_{t_i})$ lie in the $\frac {\ee}2$-neighborhood of the plane $\Si_{\infty}$ for large $t_i$. If $\Sigma_{\infty}$ has multiplicity more than one,  the space between the top and bottom sheets of $(\Si_{t_i}\cap B_{\ee^{-1}}(0))\b \mathbf{H}(\ee, \Si_{t_i})$ belongs to $\mathbf{TN}$. Thus, by the definition of $\mathbf{TN}$ we have that $\mathbf{TN}(\epsilon, \Si_{t_i}) $ is nonempty and
$$|\mathbf{TN}(\epsilon, \Si_{t_i})|>0.$$

\end{proof}

However, in general we do not know whether $|\mathbf{TN}(\epsilon, \Sigma_t)|$ is a continuous function of $t$, but we can choose $t_i$ carefully such that $|\mathbf{TN}(\epsilon, \Si_{t})|$ is bounded on a time interval. Combining Lemma \ref{lma:GD18_1} with Lemma \ref{lma:GD18_3}, we have

\begin{lem}\label{lma:GC05_3}
There is a sequence of times $t_i \to \infty$ such that
\begin{align}
  \sup_{t_i \leq t \leq t_i + i} |\mathbf{TN}(\epsilon, \Si_{t})| \leq 2 |\mathbf{TN}(\epsilon, \Si_{t_i})|.
  \label{eqn:GC05_4}
\end{align}
\end{lem}

\begin{proof}By Lemma \ref{lma:GD18_3} we can find $s_1\geq 1$ such that  \beq |\mathbf{TN}(\epsilon, \Si_{s_1})|>0. \label{eq:Z005}\eeq
We search for time $t \in [s_1, s_1+1]$ satisfying $|\mathbf{TN}(\epsilon, \Si_t)|> 2|\mathbf{TN}(\epsilon, \Si_{s_1})|$. If no such time exists, then we set
  $t_1=s_1$. Otherwise, we choose such a time and denote it by $s_1^{(1)}$. Then we search the time interval $[s_1^{(1)}, s_1^{(1)}+1]$. Inductively, we search
  $[s_1^{(k)}, s_1^{(k)}+1]$. If we have
  \begin{align*}
    \sup_{t \in [s_1^{(k)}, s_1^{(k)}+1]} |\mathbf{TN}(\epsilon, \Si_t)| \leq 2 |\mathbf{TN}(\epsilon, \Si_{s_1^{(k)}})|,
  \end{align*}
  then we denote $t_1=s_1^{(k)}$ and stop the searching process. Otherwise, we choose a time $s_1^{(k+1)} \in [s_1^{(k)}, s_1^{(k)}+1]$ with more than doubled
  $|\mathbf{TN}|$ value and continue the process. Note that
  \begin{align*}
    |\mathbf{TN}(\epsilon, \Si_{s_1^{(k)}})| \geq 2^{k} |\mathbf{TN}(\epsilon, \Si_{s_1})| \to \infty,  \quad \textrm{as} \quad k \to \infty,
  \end{align*} where we used (\ref{eq:Z005}).
  Since by Lemma \ref{lma:GD18_1} $|\mathbf{TN}(\epsilon, \Sigma_{t})| $ tends to zero , this process must stop in finite steps.
  After we find $t_1$, set $s_2^{(0)}=t_1 +1$ and continue the previous process to find time in $[s_2^{(0)}, s_2^{(0)}+2]$ with more than
  doubled $|\mathbf{TN}|$-value, with slight change that the time-interval has length $2$. Similarly, for some finite $k$, we have
  \begin{align*}
      \sup_{t \in [s_2^{(k)}, s_2^{(k)}+2]} |\mathbf{TN}(\epsilon, \Si_t)| \leq 2 |\mathbf{TN}(\epsilon, \Si_{s_2^{(k)}})|.
  \end{align*}
  Then we define $t_2=s_2^{(k)}$.
  Inductively, after we find $t_l$, we set $s_{l+1}^{(0)}=t_{l}+l$. Then we start the process to search time in $[s_{l+1}^{(0)}, s_{l+1}^{(0)}+l+1]$ with
  more than doubled $|\mathbf{TN}|$ value.
  This process is well defined. From its construction, it is clear that $t_i \to \infty$ and satisfies (\ref{eqn:GC05_4}). The lemma is proved. \\
\end{proof}

Let $O(3)$ be the group of all rotations in $\RR^3$.  It is clear that $O(3)$ is compact and preserves many geometric quantities including distance,  mean curvature and second fundamental form.
Suppose $\mathbf{x}(t)$ is a rescaled mean curvature flow solution of (\ref{eq:BB001}), then $\sigma \circ \mathbf{x}$ is again a solution of (\ref{eq:BB001}) for each fixed $\sigma \in O(3)$.
Therefore, up to rotations, we can always assume the limit plane in Lemma ~\ref{lem:A001} to be a horizonal plane.  For if $\Sigma_{i,t}$ converges to some plane $P'$ passing through the origin, we can choose
a rotation $\sigma$ such that $\sigma(P')$ is a horizontal plane.  Then it is clear that $\sigma \circ \Sigma_{i,t}$ converges to the horizontal plane $\sigma(P')$.
It is also important to note that the quantities defined in Definition~\ref{def:GD18_1} are invariant under the action of  $\sigma \in O(3)$.
Consequently,  for the flow $\sigma \circ \Sigma_{i,t}$,  the estimate (\ref{eqn:GC05_4}) becomes
 \begin{align*}
      \sup_{t_i \leq t \leq t_i + i} |\mathbf{TN}(\epsilon, \sigma \circ \Si_{t})| \leq 2 |\mathbf{TN}(\epsilon, \sigma \circ \Si_{t_i})|.
 \end{align*}
Therefore, replacing $\Sigma_{i,t}$ by $\sigma \circ \Sigma_{i,t}$ if necessary, we can always assume the limit plane is horizontal and (\ref{eqn:GC05_4}) holds.

\subsection{Construction of auxiliary functions}
\label{subsec:auxiliary}

Let $P$ be the horizontal plane
\begin{align}
    P:=\{(x_1, x_2, x_3)\;|\,x_3=0\}\subset \RR^3.    \label{eqn:PE14_1}
\end{align}
For each pair of fixed positive numbers $r_1, r_2$ with $r_1<r_2$,  we define
\begin{align}
  \Ann(r_1, r_2) :=   \Big( B_{r_2}(0) \backslash B_{r_1}(0)\Big) \cap P.
\end{align}
Based on the choice of $\epsilon$ and subsequences in Lemma~\ref{lma:GC05_3}, on $\Ann(\epsilon, \epsilon^{-1}) \times [0,\infty)$,  we shall construct an auxiliary smooth function $w$
which will be used to show the $L$-stability of the plane $P$ in the next subsection(c.f. Lemma~\ref{lma:GB24_2} and Lemma~\ref{lem:Z}).
The existence of such $w$,  which is the main result of this subsection, is provided by the the following proposition.

\begin{prop}\label{prn:PE14_1}
Under the assumption of Theorem \ref{theo:removable}, if the multiplicity of the convergence is at least two, then for each sufficiently small $\epsilon>0$, we can find a smooth function $w$ defined on $\Ann(\epsilon, \epsilon^{-1}) \times [0, \infty)$ such that
 \begin{align}
   \frac{\partial w}{\partial t}=L w :=\Delta_0 w -\frac{1}{2} \langle x, \nabla w\rangle + \frac{1}{2} w,    \label{eqn:PE14_2}
 \end{align} where $\Delta_0$ is the standard Laplacian operator on the plane $P.$
 Furthermore, there is a constant $C>0$ independent of $t$ such that
\beqn
  0<w(x, t) &<& C, \quad \forall \; (x,t) \in \Ann(\epsilon, \epsilon^{-1}) \times [0, \infty),     \label{eqn:PE14_3}\\
  w(x, 0)&>& \frac 1C, \quad \forall \; x \in \Ann(\epsilon, \epsilon^{-1}).      \label{eqn:PE14_15}
\eeqn\\

\end{prop}

The function $w$ in (\ref{eqn:PE14_2}) is the renormalized ``height-difference" function.  Its existence boils down to the discussion in subsection~\ref{subsec:conv} and subsection~\ref{subsec:decompose}.
In summary,  before we prove Proposition~\ref{prn:PE14_1},  we already have the following properties  in the remainder discussion in this subsection.

\begin{enumerate}
\item[(1).] $\epsilon$ is fixed.
\item[(2).] The sequence $\{t_i\}$ is chosen as in Lemma \ref{lma:GC05_3}. Thus, the volume of $\mathbf{TN}$ increases slowly in the sense that
      \begin{align}
      \sup_{t_i \leq t \leq t_i + i} |\mathbf{TN}(\epsilon, \Si_{t})| \leq 2 |\mathbf{TN}(\epsilon, \Si_{t_i})|.
       \label{eqn:PE14_14}
     \end{align}

\item[(3).] For any $T>0$, there is a subsequence, still denoted by $\{t_i\}$, such that $\{\Sigma_{i,t}, -T<t<T\}$ converges to a limit flow,  with multiplicity $N_0$, on the plane
      \begin{align}
         \Si_{\infty}=P=\{(x_1, x_2, x_3)\;|\,x_3=0\}\subset \RR^3.    \label{eqn:PE09_1}
      \end{align} Here we assumed $\Si_{i, t}=\Si_{t_i+t}$ as before.
The convergence is smooth away from $\cS=\{(x, t)\;|\;x\in \cS_t, \;t\in (-T, T) \}$, which satisfies
         \begin{align}
           \cS_t=e^{\frac t2}\cS_0,     \label{eqn:PE14_4}
         \end{align}
The set $\mathcal{S}_0=\left\{ p_1, p_2, \cdots, p_l\right\}$ satisfies
         \begin{align}
          l \leq M_0,    \label{eqn:PE14_13}
         \end{align} where $M_0$ is the constant in the right-hand side of (\ref{eqn:PE10_7}).
\item[(4).] By Lemma \ref{lem:A001}, the limit plane $P$ and $\cS_0$ are independent of the choice of $T.$ In other words,
for different $T$  we have different subsequences of $\{\Sigma_{i,t}, -T<t<T\}$ but these two subsequences converge to the same limit plane $P$ with the same  $\cS_0$.
\end{enumerate}

The following lemma is a direct consequence of (\ref{eqn:PE14_4}).
\begin{lem}\label{lem:t}
 There exists $t_{\ee}>0$ depending only on $\ee$ and $\cS_0$ such that
 \begin{align}
   \Ann(\epsilon, \epsilon^{-1}) \cap \cS_t=\emptyset, \quad \forall \; t > t_{\epsilon}.  \label{eqn:PE14_10}
 \end{align}
Furthermore, if $\;0 \notin \mathcal{S}_0$, then (\ref{eqn:PE14_10}) can be improved as
 \begin{align}
  \Big(B_{\epsilon^{-1}}(0) \cap P\Big) \cap \cS_t=\emptyset, \quad \forall \; t > t_{\epsilon}.  \label{eqn:PE14_11}
 \end{align}

\end{lem}

\begin{proof}

In light of (\ref{eqn:PE14_4}), all points in $\cS_t\b \{0\}$ will move further away from the point $\{0\}$ when $t$ is increasing.
Set
 \beq \min \cS_0:=\min\{|y|\,\;|\; y\in \cS_0\b \{0\}\}>0,\quad t_{\ee}:=2\log \frac 1{\ee \min\cS_0}.\label{eqB:ti}
\eeq
Then for any $t\geq t_{\ee}$ and any $y_0\in  \cS_0\b\{0\}$ we have $e^{\frac t2}|y_0|\geq\ee^{-1}$. This implies that
\begin{align}
(\cS_t\b \{0\})\cap B_{\ee^{-1}}(0)=\emptyset, \quad \forall \; t>t_{\epsilon}.\label{eqC:001}
\end{align}
Then both (\ref{eqn:PE14_10}) and (\ref{eqn:PE14_11}) follow from (\ref{eqC:001}).
\end{proof}

For the convergence from $\Sigma_{i,t}$ to $\Sigma_{\infty}$, it is clear that the origin $0$ is a natural base point.
It is possible that $0 \in \mathcal{S}_0$.   If this is the case, then $0 \in \mathcal{S}_t$ for each $t$, which is very different from all other singular points in $\mathcal{S}_0$ by Lemma~\ref{lem:t}.
Therefore, if $0 \in \mathcal{S}_0$, it is essential in the sense that it cannot be perturbed away by varying $t$.
However, for later application, we prefer to choose a base point uniformly away from singular set, whose existence is  guaranteed  by the following lemma.

\begin{lem}\label{lma:PE14_1}
 There exists a point $x_0 \in \Ann(2\epsilon,1)$ such that
 \begin{align}
  (B_{4\ee}(x_0)\cap \Si_{\infty})\cap \cS_t=\emptyset, \quad \forall\; t\in (-\infty, \infty). \label{eq:Y001}
 \end{align}

\end{lem}

\begin{proof}

Denote by $\gamma_j=\{sp_j\;|\; s\geq 0\}$, where $p_j$ is any point in $\cS_0=\{p_1, p_2, \cdots, p_l\}.$
It  is clear that $\gamma_j$ is the ray  passing through $p_j$ with the initial point $0$ if $p_j \neq 0$ and $\gamma_j=\{0\}$ if  $p_j=0$.
Recall that the number of points in $\cS_0$ is uniformly bounded by (\ref{eqn:PE10_7}) and $\epsilon$ is very small,  we can find a point $x_0\in \Ann(2\ee, 1)$ such that
\beq (B_{4\ee}(x_0)\cap \Si_{\infty})\cap (\cup_{j=1}^l \gamma_j)=\emptyset.\label{eq:X003}\eeq
One can see Figure~\ref{fig:regularbase} for intuition, where the case $l=3$ was illustrated.
Then (\ref{eq:Y001}) follows from the combination of (\ref{eqn:PE14_4}) and (\ref{eq:X003}).
\end{proof}

 \begin{figure}
 \begin{center}
 \psfrag{p1}[c][c]{$p_1 \in \gamma_1$}
 \psfrag{p2}[c][c]{$p_2 \in \gamma_2$}
 \psfrag{p3}[c][c]{$p_3 \in \gamma_3$}
 \psfrag{p4}[c][c]{$p_1 \in \gamma_4$}
 \psfrag{c1}[c][c]{$B_{4\epsilon}(x_0) \cap P$}
 \psfrag{c2}[c][c]{$B_{\epsilon}(0) \cap P$}
 \psfrag{c3}[c][c]{$B_{1}(0) \cap P$}
 \psfrag{O}[c][c]{$\textcolor{red}{0}$}
 \psfrag{x0}[c][c]{$\textcolor{blue}{x_0}$}
 \includegraphics[width=0.5 \columnwidth]{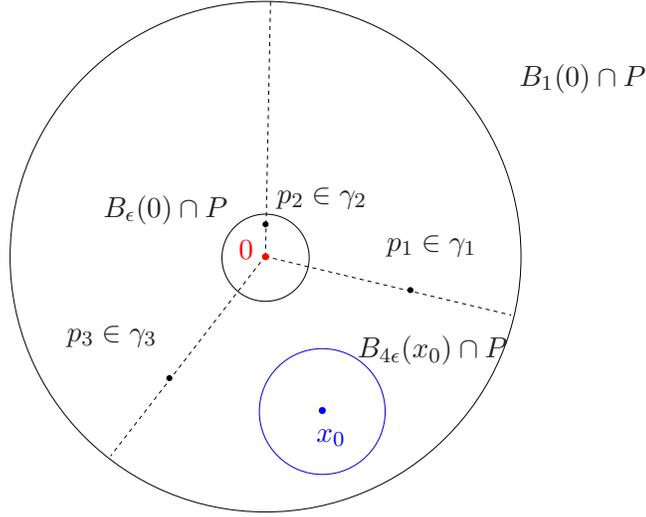}
 \caption{An alternate base point of the limit plane}
 \label{fig:regularbase}
 \end{center}
 \end{figure}

Fix an interval $I=(-T,T)$ and choose $t \in I$.
By (\ref{eq:Y001}), it is clear that the convergence from $\Sigma_{i,t}$ to $\Sigma_{\infty}$ around point $x_0$ is smooth with multiplicity $N_0$.
Let $\alpha_{x_0}$ be the vertical line passing through $x_0 \in P$. For large $i$, we have
\begin{align}
   \Sigma_{i,t} \cap \alpha_{x_0}=\left\{z_{i,t}^{(1)}, z_{i,t}^{(2)}, \cdots, z_{i,t}^{(N_0)} \right\},    \label{eqn:PE16_1}
\end{align}
where the points $z_{i,t}^{(j)}$ are ordered increasingly by their $x_3$-coordinates. Namely, we have
\begin{align}
   x_3\left( z_{i,t}^{(1)} \right) <  x_3\left( z_{i,t}^{(2)} \right)  < \cdots < x_3\left( z_{i,t}^{(N_0)} \right).   \label{eqn:PE16_2}
\end{align}
Note that near $x_0$, the smooth convergence guarantees that the normal directions of $\Sigma_{i,t}$ are almost perpendicular to the horizontal plane $P$.
Therefore, the vertical line $\alpha_{x_0}$ intersects $\Sigma_{i,t}$ transversally.
Together with the embeddedness of $\Sigma_{i,t}$,  this forces that all the inequalities in (\ref{eqn:PE16_2}) are strict.
By Lemma~\ref{lma:PE16_1}, the multiplicity is independent of time. Therefore,  for large $i$, (\ref{eqn:PE16_2}) holds for each $t \in (-T, T)$.
In particular, we have
\begin{align}
   x_3\left( z_{i,t}^{(j)} \right) <x_3\left( z_{i,t}^{(k)} \right),  \; \forall \; t \in (-T, T) ; \quad     \Leftrightarrow  \quad x_3\left( z_{i,0}^{(j)} \right) <x_3\left( z_{i,0}^{(k)} \right).
\label{eqn:PE16_3}
\end{align}
We define
\begin{align}
  &\Om_{\ee}(t) := (\Si_{\infty}\cap B_{\ee^{-1}}(0))\backslash \cup_{p\in \cS_t}B_{\ee}(p), \label{eqn:PE09_2}\\
  &\hat \Om_{\ee}(t) :=\{(x_1, x_2, x_3)\in \RR^3\;|\; (x_1, x_2, 0)\in \Om_{\ee}(t), x_3\in \RR\}, \label{eqn:PE09_3}\\
  &\Om_{\ee}(I) := \bigcap_{t\in I}\Om_{\ee}(t).  \label{eqn:PE09_4}
\end{align}
By its definition in (\ref{eqn:PE09_2}), it is clear that the convergence around $\Om_{\ee}(t)$ is smooth.  Therefore, $\Si_{i, t}\cap \hat{\Omega}_{\epsilon}(t)$ is a multi-sheet graph of exactly
$N_0$ sheet(c.f. Lemma~\ref{lma:PE16_1}).  We define the connected component of $\Si_{i, t}\cap \hat{\Omega}_{\epsilon}(t)$ to be the $j$-th sheet if this component contains the point $z_{i,t}^{(j)}$.
By (\ref{eqn:PE16_2}), we know that this is well defined and the choice of $j$ is independent of time.
Furthermore, each component of $\Si_{i, t}\cap \hat{\Omega}_{\epsilon}(t)$ is a graph of smooth functions $u_i^{(j)}(\cdot, t)$ over $\Omega_{\epsilon}(t)$ such that
\begin{align}
   u_i^{(1)}(q,t)< u_i^{(2)}(q,t) < \cdots <u_i^{(N_0)}(q,t), \quad \forall \;  q \in  \Omega_{\epsilon}(t), \; t \in (-T, T).    \label{eqn:PE14_7}
\end{align}
According to their construction, it is clear that
\begin{align*}
    u_i^{(j)}(x_0',t)=x_3\left( z_{i,t}^{(j)} \right), \quad \forall \; j \in \{1, 2, \cdots, N_0\},
\end{align*}
where $x_0'$ is the orthogonal projection image of $x_0$ from $\RR^3$ to the horizontal plane $P=\Sigma_{\infty}$.
The $j$-th component of  $\Si_{i, t}\cap \hat\Om_{\ee}(t)$ is called top sheet if $j=N_0$ and it is called bottom sheet if $j=1$.
For simplicity of notation, let $u_i^+(x, t)$ and $u_i^-(x, t)$ be the functions representing the top and bottom sheets, i.e.,
\begin{align}
   u_i^+(x, t) := u_i^{(N_0)}(x,t), \quad u_i^-(x, t) := u_i^{(1)}(x,t).    \label{eqn:PE14_8}
\end{align}
Then we denote the graphs of $u_i^{\pm}(x, t)$ over $\Omega_{\epsilon}$ by $\Si_{i, t}^{\pm}$ respectively.
It is not hard to see that $u_i^{\pm}(x, t)$ satisfy an evolution equation determined by the rescaled mean curvature flow.
We write down the details for the convenience of the readers.

\begin{lem}\label{lem:u} The functions $u_i^+(x, t)$ and $u_i^-(x, t)$ satisfy the equation on $\Om_{\ee}(I)\times I$
\beq
\pd ut=\Delta_0u-\frac {\Na^2 u(\Na u, \Na u)}{1+|\Na u|^2}
-\frac 12\langle x, \Na u\rangle+\frac u2. \label{eqn:002}
\eeq
\end{lem}
\begin{proof} Let $\Omega\subset P=\{(x_1, x_2, x_3)\;|\,x_3=0\}$. Consider the graph of the function $u: \Omega\times I\ri \RR$
$$\mathrm{Graph}_u=\{(x_1, x_2, u(x_1, x_2, t))\;|\; (x_1, x_2)\in \Omega, t\in I\}.$$
The upward pointing unit normal and the mean curvature are given by
\beq \n=\frac {(-u_{x_1}, -u_{x_2}, 1)}{\sqrt{1+|\Na u|^2}} ,\quad H=-\div\Big(\frac {\Na u}{\sqrt{1+|\Na u|^2}}\Big) \label{eqn:003}\eeq
respectively.
Thus, we have
\beq
H-\frac 12\langle \x, \n\rangle=-\div\Big(\frac {\Na u}{\sqrt{1+|\Na u|^2}}\Big)+\frac {x_1 u_{x_1}+x_2u_{x_2}-u}{2\sqrt{1+|\Na u|^2}}, \label{eqn:005}
\eeq where $\x=(x_1, x_2, u(x_1, x_2, t))$ denotes a point on $\mathrm{Graph}_u.$ On the other hand,  we have
\beq \langle \pd {\x}t, \n\rangle=\pd ut \frac 1{\sqrt{1+|\Na u|^2}}. \label{eqn:004} \eeq
Combining (\ref{eqn:005}) and (\ref{eqn:004}) with the equation of rescaled mean curvature flow (\ref{eq:RMCF0}), we have
\beq \pd ut=\sqrt{1+|\Na u|^2}\div\Big(\frac {\Na u}{\sqrt{1+|\Na u|^2}}\Big)-\frac 12 \langle x, \Na u \rangle+\frac u2, \label{eqn:100} \eeq where $x=(x_1, x_2, 0)\in \Omega$ and $\langle x, \Na u \rangle=x_1u_{x_1}+x_2u_{x_2}$.
Direct calculation shows that (\ref{eqn:100}) implies (\ref{eqn:002}).
Note that $u_i^+(x, t)$ and $u_i^-(x, t)$ denote the graph functions of $\Si_{i, t}^+$ and $\Si_{i, t}^-$ respectively. Therefore, $u_i^+(x, t)$ and $u_i^-(x, t)$  satisfy (\ref{eqn:002}).  The lemma is proved.
\end{proof}

Let  $u_i=u_i^+-u_i^-$ and $f_{i, s}= u_{i}^-+s u_i$ for $s\in [0, 1]$. Then $f_{i, 0}=u_{i}^-$ and $f_{i, 1}=u_i^+$. We define
$$G(s):=\frac {\Na^2 f_{i, s}(\Na f_{i, s}, \Na f_{i, s})}{1+|\Na f_{i, s}|^2}.$$
Therefore, we have
\beq
G(1)-G(0)=\int_0^1\, \pd {G(s)}s\,ds=a^{pq}u_{i, pq}+b^pu_{i, p}, \label{eq:F004}
\eeq where
\beqn
a^{pq}&=&\int_0^1\,\Big((1+|\Na f_{i, s}|)^{-1}\Na_p f_{i, s}\Na_qf_{i, s}\Big)\;ds, \label{eq:F005}\\
b^p&=&\int_0^1\,\Big(2(1+|\Na f_{i, s}|)^{-1}\Na_k\Na_pf_{i, s}\Na_k f_{i, s}\nonumber\\&&- 2(1+|\Na f_{i, s}|)^{-2}\Na_pf_{i, s} \Na^2 f_{i, s}(\Na f_{i, s}, \Na f_{i, s})\Big)\,ds.\label{eq:F006}
\eeqn
Combining (\ref{eqn:002}) with (\ref{eq:F004}), we get the equation of $u_i$
\beq
\pd {u_i}t=\Delta_0u_i-\frac 12\langle x, \Na u_i\rangle+\frac {u_i}2-a^{pq}u_{i, pq}-b^pu_{i, p} \label{eq:F008}
\eeq for any $(x, t)\in \Om_{\ee}(I)\times I$.
Since $\Si_{i, t}^+$ and $\Si_{i, t}^-$ converge locally smoothly to $\Si_{\infty}$ as $i\ri +\infty$, all derivatives of $u_i^+$ and $u_i^-$ converge to zero on $\Om_{\ee}(I)\times I$.
Therefore, by the expression (\ref{eq:F005})-(\ref{eq:F006}) all $C^{k, \al}$ norms of
$a^{pq}$ and $b^p$ are uniformly bounded when $t_i$ is large and tends to zero as $t_i\ri +\infty.$

Clearly, $u_i$ are positive solutions satisfying $u_i \to 0$.  We shall normalize $u_i$ to obtain $w_i$ which satisfies the linearization of the equation satisfied by $u_i$.
For this purpose, we need a space-time base point with uniform regularity.
Note that the space-time singular set $\mathcal{S}$ has good regularity (\ref{eqn:PE14_4}) and the number of points in $\cS_0$ is uniformly bounded by  (\ref{eqn:PE14_13}).
It is not hard to see that
\beq
\Om_{\frac {\ee}8}(0)\subset \Om_{\frac {\ee}{10}}(t),\quad \forall\; t\in [-2\ee', 2\ee'] \label{lma:PE14_2}
\eeq
for some $\ee'=\ee'(\cS_0, \ee)>0$.
Then $(x_0, \epsilon')$ is a good space-time base point.
We now normalize the function $u_i$ by
 \beq w_i(x, t) :=\frac {u_i(x, t)}{u_i(x_0, \ee')}, \quad \forall\;(x, t)\in \Om_{\ee}(I)\times I.  \label{eq:F001}\eeq
Then $w_i(x, t)$ is a positive function on $ \Om_{\ee}(I)\times I$ satisfying $w_i(x_0, \ee' )=1$ and
\beq
\pd {w_i}t=\Delta_0w_i-\frac 12\langle x, \Na w_i\rangle+\frac {w_i}2-a^{pq}w_{i, pq}-b^pw_{i, p},  \label{eq:F007}
\eeq where $a^{pq}$ and $b^p$ are defined by (\ref{eq:F005}) and (\ref{eq:F006}).\\

The estimate of $|\mathbf{TN}|$ in (\ref{eqn:PE14_14}) provides information on the growth rate of $w_i$, through the application of the following Lemma.
It basically means that the volume of the thin part is almost the volume of the region between the top and bottom sheets.

\begin{lem}\label{lem:equi}
For large $t_i$, we have
\beq
\int_{\Om_{\ee}(t)}\, u_i(x, t) \,d\mu_{\infty}\leq  |\mathbf{TN}(\epsilon, \Sigma_{i, t})|\leq \int_{\Om_{\frac {\ee}5}(t)}\, u_i(x, t) \,d\mu_{\infty},\label{eq:J001}
\eeq where $d\mu_{\infty}$ denotes the standard volume form of $\Si_{\infty}$.

\end{lem}
\begin{proof}
For each $t$ and large $t_i$, we define
$$ \Phi(\ee, \Si_{i, t})=\{(x_1, x_2, x_3)\in \RR^3\;|\; x':=(x_1, x_2, 0)\in \Om_{\ee}(t), u_i^-(x', t)\leq x_3\leq u_i^+(x', t)\}. $$
Then the Euclidean volume of $\Phi(\ee, \Si_{i, t})$ can be calculated by
\beq |\Phi(\ee, \Si_{i, t})|=\int_{\Om_{\ee}(t)}\,u_i(x, t)\,d\mu_{\infty}. \label{eq:I002}\eeq
We claim that
\beq
\Phi(\ee, \Si_{i, t})\subset \mathbf{TN}(\epsilon, \Si_{i, t})\subset \Phi(\frac {\ee}5, \Si_{i, t}). \label{eq:I003}
\eeq
  One can check Figure~\ref{fig:multisheet} for intuition, where the case $\mathcal{S}_t=\{p_1\}$ is addressed
and $p_1'$ is the $(x_1, x_2)$ coordinate of the point $p_1$.

 \begin{figure}
 \begin{center}
 \psfrag{A}[c][c]{$x_3$}
 \psfrag{B}[c][c]{$x'=(x_1,x_2)$}
  \psfrag{C}[c][c]{$|x'-p_1'|=\frac{\epsilon}{5}$}
 \psfrag{D}[c][c]{$\textcolor{blue}{|x'-p_1'|=\epsilon}$}
 \psfrag{E}[c][c]{$x_3=u_i^{+}(x',t)$}
 \psfrag{F}[c][c]{$x_3=u_i^{-}(x',t)$}
 \psfrag{H}[c][c]{$\textcolor{red}{\mathbf{H}(\epsilon, \Sigma_{i,t})}$}
 \psfrag{p1}[c][c]{$p_1$}
 \includegraphics[width=0.5 \columnwidth]{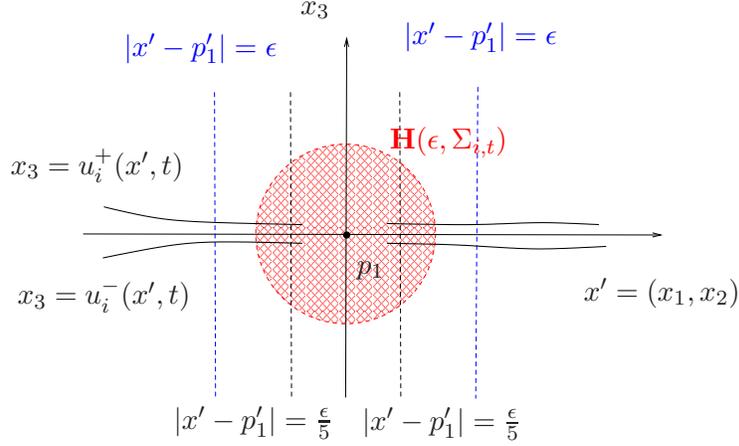}
 \caption{Thin part is almost the region between the top and bottom sheets}
 \label{fig:multisheet}
 \end{center}
 \end{figure}

Now we prove (\ref{eq:I003}) rigorously. Note that $\Si_{i, t}$ converges locally smoothly to $\Si_{\infty}$ away from $\cS_t$.
For each $t$, the set $\mathbf{S}(\ee, \Si_{i,t})$ converges to $\mathcal{S}_t$.
So for large $i$, we have
$$\mathbf{S}(\ee, \Si_{i, t})\subset B_{\frac {\ee}{10}}(\cS_t)=\cup_{p\in \cS_t}B_{\frac {\ee}{10}}(p).$$
By the definition of $\mathbf{H}$ in Definition~\ref{def:GD18_1}, we have
\beq \mathbf{H}(\ee, \Si_{i, t})\subset B_{\frac 34\ee}(\cS_t)\cap B_{\ee^{-1}}(0).\label{eq:K006}\eeq
For any $x=(x_1, x_2, x_3)\in \Phi(\ee, \Si_{i, t})$, we have
$(x_1, x_2, 0)\in \Om_{\ee}(t)$, which implies that $(x_1, x_2, x_3)\not\in B_{\ee}(\cS_t)$. Therefore, $x\not\in  \mathbf{H}(\ee, \Si_{i, t})$ and we have $ \Phi(\ee, \Si_{i, t})\subset \mathbf{TN}(\epsilon, \Si_{i, t}) $ when $t_i$ is large.  On the other hand,
for any $x\in \mathbf{TN}(\epsilon, \Si_{i, t})$, we have
$x\not\in \mathbf{TK}(\ee, \Si_{i, t})$ and $x\not\in \mathbf{H}(\ee, \Si_{i, t})$, which implies that $d(x, \mathbf{S})\geq\frac {\ee}2$ by the definition of $\mathbf{H}$. Since $\cS_t$ consists of finite many points, $d(x, \cS_t)$ is attained by some point $q_0\in \cS_t $.
Thus, we have
\beq d(x, \cS_t)=\min_{y\in \cS_t}d(x, y)=d(x, q_0)\geq \max_{z\in \mathbf{S}\cap B_{\frac {\ee}{10}}(q_0)}(d(x, z)-d(z, q_0))\geq \frac {\ee}4, \label{eq:K004}\eeq
where we used the fact that $d(x, z)\geq \frac {\ee}2$ and $d(z, q_0)\leq \frac {\ee}{10}.$ Moreover, since $x\not\in \mathbf{TK}(\ee, \Si_{i, t})$, $x$ is sufficiently close to $\Si_{\infty}$. This together with (\ref{eq:K004})
implies that $x\in \Phi(\frac {\ee}5, \Si_{i, t}) $ when $i$ is large. Therefore, (\ref{eq:I003}) is proved.
Thus, the inequality (\ref{eq:J001}) follows from (\ref{eq:I002}) and (\ref{eq:I003}). The lemma is proved.

\end{proof}

Combining the previous discussion, we obtain some uniform  estimates on the functions $w_i$.
\begin{lem}\label{lem:estimates}
For any compact set $K\Subset \Ann(2\ee, \frac 12\ee^{-1})$, $T>t_{\epsilon}+2$,
there exist two constants
\begin{align*}
 C_1=C_1(\ee, K, \cS_0, x_0)>0, \quad C_2=C_2(\ee, \cS_0, T,  x_0)>0
\end{align*}
such that for large $t_i$, we have
\beq C_2(\ee,  \cS_0, T, x_0)<w_i(x, t)\leq C_1(\ee, K, \cS_0, x_0),\quad \forall\, (x, t)\in K\times [t_{\ee}+1, T-1]. \label{eq:G001}\eeq
Moreover, at time $t=t_{\ee}+1$ there exists $C_3=C_3(\ee,  \cS_0,  x_0)>0$ independent of $T$ such that
\beq
w_i(x, t_{\ee}+1)\geq C_3(\ee,  \cS_0,  x_0)>0,\quad \forall\;x\in K.  \label{eq:G100}
\eeq

\end{lem}
\begin{proof}We divide the proof into several steps.\\

{\it Step 1.  $w_i(x, 0)$ is bounded from above on $K$.}
By the definition of $\ee'$ in (\ref{lma:PE14_2}), all points of $\Om_{\frac {\ee}8}(0)$ are regular for any $t\in [-2\ee', 2\ee']$. Thus, the equation (\ref{eq:F007}) of $w_i$  can be written as
\beq
\pd {w_i}t=\td a^{pq}w_{i, pq}+\td b^pw_{i, p}+\frac 12 w_i,\quad \forall \;(x, t)\in \Om_{\frac {\ee}8}(0)\times (-2\ee', 2\ee'),\label{eq:Z004}
\eeq where $\td a^{pq}=\dd^{pq}-a^{pq}$ and $\td b^p=-\frac 12 x^p-b^p. $ Here $\dd^{pq}=0$ if $p\neq q$ and $\dd^{pq}=1$ if $p=q.$ Since the sequence $\{\Si_{i, t}, -T< t< T\}$ converges locally smoothly to the plane $P$ away from the singular set $\cS,$ the coefficients $a^{pq}$ and $b^p$ tend to zero on $\Om_{\frac {\ee}8}(0)\times (-2\ee', 2\ee')$. Thus, $\td a^{pq}$ and $\td b^p$ satisfy the conditions (\ref{eqB:001})-(\ref{eqB:002}) in Appendix A.
We set
\begin{align}
  \Om :=\Om_{\frac {\ee}8}(0), \quad  \Om' :=\Om_{\frac {\ee}5}(0), \quad \Om'' :=\Om_{\frac {\ee}6}(0). \label{eqn:PE14_7}
\end{align}
It follows from the definition (\ref{eqn:PE09_2}) that we have the relationship
\begin{align}
  \Omega' \Subset \Omega'' \Subset \Omega.  \label{eqn:PE14_8}
\end{align} Since  $\overline {\Om'}$ is compact, we can cover $\Om'$ by finite many balls contained in $\Om''$ with radius $r=\frac {\ee}{100}$ and the number of these balls is bounded by a constant  depending only on $\ee$. Moreover, $\Om''$ has a positive distance $\dd=\dd(\ee)>0$ from the boundary of $\Om.$
Thus, applying Theorem \ref{theo:B3} for such $\Om, \Om', \Om''$ and the time interval $(-2\ee', 2\ee')$, we have
\beq
w_i(x, 0)\leq C(\ee, \ee', x_0)\,w_i(x_0, \ee')=C(\ee, x_0, \cS_0),\quad \forall\; x\in \Om'. \label{eq:Z001}
\eeq  Here  $C(\ee, x_0, \cS_0)$ means that $C$ depends only on $\ee, x_0, \cS_0.$ \\

{\it Step 2.   $w_i(x, t)$ is bounded from above  on $ K\times [t_{\ee}+1, T-1].$}
Integrating both sides of (\ref{eq:Z001}) on $\Om'$, we have
\beq
\int_{\Om_{\frac {\ee}5}(0)}\,w_i(x, 0)\leq C(\ee, x_0, \cS_0). \label{eq:Z002}
\eeq
By Lemma \ref{lma:GC05_3} and Lemma \ref{lem:equi}, for large $t_i$ and any $t\in  [0, T)\subset [0, i]$ we have
\beqn
\int_{\Om_{\ee}(t)}\, u_i(x, t)\,d\mu_{\infty}&\leq& |\mathbf{TN}(\epsilon, \Si_{t_i+ t})| \leq 2 |\mathbf{TN}(\epsilon, \Si_{t_i})|\nonumber\\
&\leq& 2\int_{\Om_{\frac {\ee}5}(0)}\, u_i(x, 0) \,d\mu_{\infty}.\label{eq:J005}
\eeqn
Combining (\ref{eq:J005}) with the definition (\ref{eq:F001}) of $w_i$, for large $t_i$ and any $t\in [0, T)$ we have the following inequality \beq
\int_{\Om_{{\ee}}(t)}\, w_i(x, t) \,d\mu_{\infty}\leq 2\int_{\Om_{\frac {\ee}5}(0)}\, w_i(x, 0) \,d\mu_{\infty}\leq C(\ee,  x_0, \cS_0), \label{eq:K001}
\eeq where we used (\ref{eq:Z002}) in the last inequality.  Note that by Lemma \ref{lem:t}, for $t\in [t_{\ee}, T)$ we have $\Om_{\ee}(t)=B_{\epsilon^{-1}}(0)\cap \Si_{\infty}$ if $0\notin \cS_0$ and $\Om_{\ee}(t)=\Ann(\ee, \ee^{-1})$ if $0\in \cS_0.$
Let
\beq
\Om:=\Ann(\ee, \ee^{-1}),\quad \Om':=\Ann(2\ee, \frac 12 \ee^{-1}),\quad \Om'':=\Ann(\frac 32\ee,  \frac 23\ee^{-1}).\label{eqC:002}
\eeq
Then we have $ \Om'\Subset \Om''\Subset \Om$. Moreover, $\Om'$ can be covered by finite many balls contained in $\Om''$ with radius $r=\frac {\ee}{100}$ and the number of these balls is bounded by a constant  depending only on $\ee$.
 For any $K\subset\subset \Om'$,
applying Theorem \ref{theo:B3} for such $\Om, \Om', \Om''$ and the interval $[t_{\ee}, T-1]$ as in {\it Step 1}, we have
\beq
w_i(x, t)\leq C( \ee)w_i(y, t+\frac 12),\quad \forall\; x, y\in K, \;t\in [t_{\ee}+1, T-1].\label{eq:Z003}
\eeq
Integrating both sides of (\ref{eq:Z003}) for $y\in K$,  we have
\beqn
w_i(x, t)&\leq& C( \ee, K)\int_K\;w_i(y, t+\frac 12)\,d\mu_{\infty}(y)\nonumber\\
&\leq &C(\ee, K)\int_{\Om_{\ee}(t+\frac 12)}\;w_i(y, t+\frac 12)\,d\mu_{\infty}(y)\nonumber\\
&\leq& C(\ee,  x_0, \cS_0, K),\quad \forall\; (x, t)\in K\times [t_{\ee}+1, T-1],
\eeqn where we used (\ref{eq:K001}) and the fact that
$$K\subset\subset \Ann(2\ee, \frac 12 \ee^{-1})\subset \Ann(\ee, \ee^{-1})\subset \Om_{\ee}(t),\quad \forall\; t\in [t_{\ee}, T).$$
Therefore, $w_i(x, t)$ is bounded from above.\\

{\it Step 3. $w_i(x_0, t_{\ee}+\frac 12)$ is bounded from below by a constant independent of $i$.} By (\ref{eq:Y001}) and (\ref{eqn:PE09_4}),  we have $(B_{\ee}(x_0)\cap \Si_{\infty})\subset \Om_{\ee}((-T, T)) $.
Therefore, applying Theorem \ref{theo:B4} for $(B_{\ee}(x_0)\cap \Si_{\infty})\times [0, t_{\ee}+1]$ we have
\beq
w_i(x_0, t_{\ee}+\frac 12)\geq C(\ee, x_0, \cS_0) w_i(x_0, \ee')=C(\ee, x_0, \cS_0), \label{eq:X004}
\eeq where we used $w_i(x_0, \ee')=1$. Thus, $w_i(x_0, t_{\ee}+\frac 12)$ is bounded from below.
\\

{\it Step 4. $w_i(x, t)$ has  a positive lower bound on $K\times [t_{\ee}+1, T-1]$.}
Note that by Lemma \ref{lma:PE14_1} we have  $x_0\in \Ann(2\ee, \frac 12 \ee^{-1})$.  We set
$\Om, \Om'$ and $\Om''$ as in (\ref{eqC:002}) and we have $K\cup\{x_0\}\subset\Om'$.
  Thus, we apply Theorem \ref{theo:B3} for such $\Om, \Om'$ and the interval $[t_{\ee}, T-1]$ to get
\beq
w_i(x, t)\geq C(\ee, x_0, t-(t_{\ee}+\frac 12), \cS_0)w_i(x_0, t_{\ee}+\frac 12),\quad \forall\; (x, t)\in K\times [t_{\ee}+1, T-1]. \label{eq:X001}
\eeq  Combining (\ref{eq:X001}) with (\ref{eq:X004}), we have
\beq
w_i(x, t)\geq C(\ee, x_0, T, \cS_0),\quad  \forall\; (x, t)\in K\times [t_{\ee}+1, T-1].
\eeq In particular, if we take $t=t_{\ee}+1$ in (\ref{eq:X001}), then we have
\beq
w_i(x, t_{\ee}+1)\geq C(\ee, x_0,  \cS_0) w_i(x_0, t_{\ee}+\frac 12)\geq C(\ee, x_0,  \cS_0),\quad \forall\;x\in K. \label{eq:X002}
\eeq
Note that the constant $C$ in (\ref{eq:X002}) doesn't depend on $T$. The lemma is proved.

\end{proof}

\begin{lem}\label{lem:convergence} The same conditions as in Lemma \ref{lem:estimates}.  As $t_i\ri +\infty$, we can take a subsequence of the functions  $w_i(x, t)$ such that it converges in $C^2$ topology on $K\times [t_{\ee}+1, T-1]$, where  $K$ is any compact set satisfying $K\subset\subset \Ann(2\ee, \frac 12\ee^{-1})$,  to a positive function $w(x, t)$ satisfying
\beq \pd wt=Lw:=\Delta_0 w-\frac 12\langle x, \Na w\rangle+\frac 12 w,\quad \forall\;(x, t)\in K \times [t_{\ee}+1, T-1]. \label{eq:D004}\eeq
Moreover, there exist constants $C_1(\ee, K, \cS_0, x_0), C_2(\ee,  \cS_0, T, x_0)>0$ such that
\beq
C_2(\ee,  \cS_0, T, x_0)\leq w(x, t)\leq C_1(\ee, K, \cS_0, x_0),\quad \forall\; (x, t)\in K\times [t_{\ee}+1, T-1],\label{eq:K002}
\eeq
and  at time $t=t_{\ee}+1$ there exists $C_3=C_3(\ee,  \cS_0,  x_0)>0$ independent of $T$ such that
\beq
w(x, t_{\ee}+1)\geq C_3(\ee,  \cS_0,   x_0)>0,\quad \forall\;x\in K.  \label{eq:K003}
\eeq
\end{lem}
\begin{proof}
Let $I=[t_{\ee}+1, T-1]$.
By Lemma \ref{lem:estimates}, the function $w_i(x, t)$ is  bounded from above and below by some constants independent of $i$. Since the function $w_i$ satisfies the equation (\ref{eq:Z004}) on any $K\times I$,
we can use the interior H\"older estimate (c.f. Theorem 4.9  of \cite{[Lieb]}, or Theorem \ref{theo:C1} in Appendix B) to get the    local space-time $C^{2, \al}$ estimates,
which implies that a subsequence converges uniformly in $C^2$  to a positive function $w(x, t)$ on $(x, t)
\in K\times I$ satisfying (\ref{eq:D004})-(\ref{eq:K003}). The lemma is proved.
\end{proof}

Now we are able to finish the proof of Proposition~\ref{prn:PE14_1}.

\begin{proof}[Proof of Proposition~\ref{prn:PE14_1}]
Note that by Property (4) in the beginning of this subsection, $P$ and $\cS_0$ are independent of $T$, and by Lemma \ref{lem:t} we see that
 $t_{\epsilon}$ is also independent of $T$. Moreover, (\ref{eq:K002}) implies that the space-times $C^{2, \al}$ norm of $w$ is bounded by a constant
 independent of $T$ on $K\times [t_{\ee}+1, T-1]$ (c.f. Theorem 4.9  of \cite{[Lieb]}, or Theorem \ref{theo:C1} in Appendix B).
  Thus,  we can take $T \to \infty$ in Lemma \ref{lem:convergence} and obtain a function, still denoted by $w$,  defined on $\Ann(3\epsilon, \frac{1}{3} \epsilon^{-1}) \times [t_{\epsilon}+1, \infty)$
satisfying (\ref{eq:D004}), (\ref{eq:K002}) and (\ref{eq:K003}).
Replacing $\epsilon$ by $\frac{1}{3}\epsilon$ and $t$ by $t-t_{\frac{\epsilon}{3}}-1$, we know that the function $w$ is defined on
$\Ann(\epsilon, \epsilon^{-1}) \times [0, \infty)$ and  satisfies (\ref{eqn:PE14_2}), (\ref{eqn:PE14_3}) and (\ref{eqn:PE14_15}). The proposition is proved.
\end{proof}

\subsection{Proof of  the multiplicity-one convergence}


In this subsection, we shall show Theorem~\ref{theo:removable}, i.e., the limit plane is multiplicity one, under the special condition $H$ decays exponentially fast.
Our argument largely follows that of Colding-Minicozzi~\cite{[CM1]}.   However, some new technical difficulties need to be addressed.   For example, the proof
of the limit surface to be $L$-stable(c.f. Lemma~\ref{lma:GB24_2}) is technically different from the classical one.  Due to the loss of general self-shinker equation,
we need to delicately choose subsequence and apply parabolic version of maximum principle.
The final contradiction relies heavily on the existence of the auxiliary function $w$ constructed in Proposition~\ref{prn:PE14_1}.


\begin{lem}   \label{lma:GB24_2}
 Let $P$ be the horizontal plane defined in (\ref{eqn:PE14_1}).
 For any  function $\phi\in W_{c}^{1, 2}(P\backslash\{0\})$,we have
\begin{align}
 -\int_{P}\,(\phi L\phi)e^{-\frac {|x|^2}{4}} \geq 0.
  \label{eqn:GB24_6}
\end{align}
Here $W_{c}^{1, 2}(P\backslash\{0\})$ denotes the set of all functions $\phi\in W^{1, 2}(P)$ with compact support in
$P\backslash \{0\},$  and $L$ is the operator defined by (\ref{eq:D004}) on the plane $P$.

\end{lem}
\begin{proof} Given a function $\phi\in W_{c}^{1, 2}(P\backslash\{0\})$, there exists $\ee>0$ such that
\begin{align*}
\overline{\Supp (\phi)}\subset \Ann(2\epsilon, \frac{1}{2}\epsilon^{-1}).
\end{align*}
For such $\ee$, we choose $\{t_i\}$ as in Lemma \ref{lma:GC05_3} and we denote by $\Si_{\infty}$ the limit plane of the sequence $\{\Si_{t_i}\}$.
Up to rotation in $\RR^3$, we may assume $\Sigma_{\infty}$ to be $P$.
In light of Proposition~\ref{prn:PE14_1}, we obtain a positive function $w$ defined on $\Ann(\epsilon, \epsilon^{-1}) \times [0, \infty)$ satisfying (\ref{eqn:PE14_2}), (\ref{eqn:PE14_3}) and (\ref{eqn:PE14_15}).
Let $v:=\log w$.  It follows from (\ref{eqn:PE14_2}) that $v$ satisfies the equation
$$\pd vt=\Delta_0 v+\frac 12-\frac 12\langle x, \Na v\rangle+|\Na v|^2,\quad \forall\; (x, t)\in  \Ann(\epsilon, \epsilon^{-1}) \times [0, \infty).$$
Then integration by parts implies that
\beqs
0&=&\int_{P}\,\div\Big(\phi^2e^{-\frac {|x|^2}{4}}\Na v\Big)\\
&=&\int_{P}\,\Big(2\phi\langle\Na \phi, \Na v\rangle+\Big(\pd vt-\frac 12-|\Na v|^2\Big)\phi^2\Big)e^{-\frac {|x|^2}{4}}\\
&\leq&\int_{P}\,\Big(|\Na \phi|^2-\frac 12\phi^2+\pd vt\phi^2\Big)
e^{-\frac {|x|^2}{4}}.
\eeqs
The above inequality can be rewritten as
\beqn
-\int_{P}\,(\phi L\phi)e^{-\frac {|x|^2}{4}}&=& \int_{P}\,\Big(|\Na \phi|^2-\frac 12\phi^2\Big)
e^{-\frac {|x|^2}{4}}\nonumber\\
&\geq&
 -\int_{P}\,\pd vt\phi^2 e^{-\frac {|x|^2}{4}}=-\frac {d}{dt}\int_{P}\, v\phi^2 e^{-\frac {|x|^2}{4}}.\label{eq:B003}
\eeqn
Fix $T>0$. Integrating both sides of (\ref{eq:B003}), we have
\beqn -\int_{0}^{T}\,dt\int_{P}\,(\phi L\phi)e^{-\frac {|x|^2}{4}}&\geq&
\int_{P}\,v\phi^2 e^{-\frac {|x|^2}{4}}\Big|_{t=0}-\int_{P}\,v\phi^2 e^{-\frac {|x|^2}{4}}\Big|_{t=T}\nonumber\\&\geq &-C,\label{eq:Z006}
\eeqn where $C$ is independent of $T$ by  (\ref{eqn:PE14_3}) and (\ref{eqn:PE14_15}) in Proposition \ref{prn:PE14_1}. Thus, we have
\beq -\int_{P}\,(\phi L\phi)e^{-\frac {|x|^2}{4}}\geq -\frac C{T}.\label{eq:G003}\eeq
Letting $T\ri +\infty$ in (\ref{eq:G003}) we obtain (\ref{eqn:GB24_6}).
The proof of Lemma~\ref{lma:GB24_2} is complete.
\end{proof}

Following the argument of Gulliver-Lawson \cite{[GL]}, we show that the plane is $L$-stable across the singular set $\cS.$
In other words, we have the following Lemma.

\begin{lem}\label{lem:Z} The same assumption as in Lemma \ref{lma:GB24_2}. For any   smooth function $\phi$ with compact support in the plane
$P,$  we have
\begin{align}
 -\int_{P}\,(\phi L\phi)e^{-\frac {|x|^2}{4}} \geq 0.
  \label{eq:G004}
\end{align}

\end{lem}

\begin{proof}
Choose $0<\dd<R<1$ and define  the function
$$ \eta(x)=\left\{
     \begin{array}{ll}
       \frac {\log R}{\log |x|}, & \quad 0<|x|<R, \\
       1, &  \quad  |x|\geq R
     \end{array}
   \right.
$$
and the function $\bb(x)=\bb(|x|)$ such that $\bb(x)=0$ for $|x|<\frac {\dd}2$, $\bb(x)=1$ for $|x|\geq \dd$, $0\leq \bb(x)\leq 1$ and $|\Na \bb|\leq \frac 3{\dd}$.  Suppose that $\phi$ is any function with compact support on $P$.
Moreover, we can check that
\beq
\int_{P}\,|\Na \eta|^2e^{-\frac {|x|^2}{4}}\leq \frac C{|\log R|}, \quad \hbox{and}
\quad \int_{P}\, |\Na \bb|^2e^{-\frac {|x|^2}{4}}\leq C,
\eeq where $C$ are universal constants.
Note that  $\eta \bb \phi$ has  compact support in $P\b \{0\}$.
Using the inequality
$$(a+b+c)^2\leq 2\Big(1+\frac 1{\tau}\Big)(a^2+b^2)+(1+\tau)c^2,\quad \forall\; \tau>0,$$
we have
\beq
|\Na(\eta \bb \phi)|^2\leq 2\Big(1+\frac 1{\tau}\Big)\phi^2\Big(\bb^2|\Na \eta|^2+\eta^2|\Na \bb|^2\Big)+(1+\tau)\eta^2\bb^2|\Na \phi|^2. \label{eq:E008}
\eeq
We normalize $\phi$ such that $|\phi|\leq 1.$ Note that $\eta \bb\phi\in W_{c}^{1, 2}(P\b \{0\})$.
Combining (\ref{eq:E008}) with Lemma \ref{lma:GB24_2}, we have
\beqs
0&\leq & -\int_{P}\,\Big(\eta \bb \phi \cdot L(\eta \bb \phi)\Big)e^{-\frac {|x|^2}{4}}\\
&=&\int_{P}\,\Big(|\Na (\eta \bb \phi)|^2-\frac 12 (\eta \bb \phi)^2\Big)e^{-\frac {|x|^2}{4}}\\
&\leq&(1+\tau)\int_{P} \Big(|\Na \phi|^2-\frac 12 \phi^2\Big)e^{-\frac {|x|^2}{4}}+C\Big(1+\frac 1{\tau}\Big)\frac 1{|\log R|}+C\Big(\frac {\log R}{\log \dd}\Big)^2\\
&&+\frac 12\int_{P}\,\Big(1+\tau-(\eta\bb)^2\Big)\phi^2\,e^{-\frac {|x|^2}{4}},
\eeqs where $C$ are universal constants. Taking $\dd\ri 0$ and then $R\ri 0$, we have \beq
0\leq (1+\tau)\int_{P} \Big(|\Na \phi|^2-\frac 12 \phi^2\Big)e^{-\frac {|x|^2}{4}}+\frac {\tau}{2}\int_{P}\,\phi^2\,e^{-\frac {|x|^2}{4}}. \nonumber
\eeq Letting $\tau\ri 0$, we get the inequality
\beq \int_{P} \Big(|\Na \phi|^2-\frac 12 \phi^2\Big)e^{-\frac {|x|^2}{4}}\geq 0. \label{eq:E005}\eeq
The lemma is proved.

\end{proof}

 The rest of the argument is the same as that in Colding-Minicozzi \cite{[CM1]}. Let $\rho: \RR\ri [0, 1]$ be a smooth cutoff function with
$\rho(s)=1$ for $s\leq R$, $\rho=0$ for $s\geq R+1 $, $0\leq \rho(s)\leq 1$ and $|\rho'(s)|\leq 2$ for $s\in \RR.$
Let $\phi(x)=\rho(|x|)$ for any $x\in P$, where $P$ is defined in (\ref{eqn:PE14_1}). Then
 Lemma \ref{lem:Z} implies that
\beqn
0&\leq&   -\int_{P}\,(\phi L\phi)e^{-\frac {|x|^2}{4}}=\int_{P} \Big(|\Na \phi|^2-\frac 12 \phi^2\Big)e^{-\frac {|x|^2}{4}}\nonumber\\&\leq&
\int_P\;\Big((\rho'(|x|))^2-\frac 12 \rho(|x|)^2\Big)\,e^{-\frac {|x|^2}{4}}\nonumber\\
&\leq& 4\int_{P\backslash B_R(0)}\, e^{-\frac {|x|^2}{4}}-\frac 12
\int_{B_R(0)\cap P}\,  e^{-\frac {|x|^2}{4}}.\label{eq:E006}
\eeqn
Note that the right hand side of (\ref{eq:E006}) is negative when $R$ is sufficiently large. Therefore, we get a contradiction.
 Thus, for the sequence of times $\{t_i\}$ in Lemma \ref{lma:GC05_3}, the multiplicity of the convergence of $\{(\Si, \x(t_i+t)), -T<t<T\}$ is one and the convergence
is smooth everywhere. Theorem \ref{theo:removable} is proved.

\section{Proof of the extension theorem}


%


For the convenience of readers, we copy down the statement of our main extension theorem, i.e.,  Theorem~\ref{theo:main1} as follows.

\begin{theo}
If
 $\x(p, t): \Si^2\ri \RR^3 (t\in [0, T))$ is a closed smooth embedded mean curvature flow
with the first  singular time $T<+\infty$, then
$\displaystyle \sup_{\Si\times [0, T)}|H|(p, t)=+\infty$.
\end{theo}

\begin{proof}
 Suppose not, we can find $(x_0, T)$ such that the mean curvature flow $\x(p, t)$ blows up at $x_0\in \RR^3$ at time $T$ with
\beq \La_0:= \sup_{\Si\times [0, T)}|H|(p, t)<+\infty. \label{eq:A106} \eeq
Then Corollary 3.6 of \cite{[Eckbook]} implies that for all $t<T$ we have
\beq
d(\Si_t, x_0)\leq 2\sqrt{T-t}, \label{eq:I005}
\eeq  where $d(\Si_t, x_0)$ denotes the Euclidean distance from the point $x_0$ to the surface $\Si_t$.  We can rescale the flow $\Si_t$ by
$$s=-\log(T-t),\quad   \td \Si_s=e^{\frac s2}\Big(\Si_{T-e^{-s}}-x_0\Big) $$
such that the flow $\{(\td\Si_s, \td \x(p, s)), -\log T\leq s< +\infty\}$ satisfies the following properties:
\begin{enumerate}
  \item[(1)]  $\td \x(p, s)$ satisfies the equation
  \beq \Big(\pd {\td \x}s\Big)^{\perp}=-\Big(\td H-\frac 12 \langle \td \x, \n\rangle\Big)\n; \label{eq:A001}\eeq
  \item[(2)] the mean curvature of $\td \Si_s$ satisfies $|\td H(p, s)|\leq \La_0 e^{-\frac s2}$ for some $\La_0>0$;
  \item[(3)] $d(\td \Si_s, 0)\leq 2$.
\end{enumerate}
By Theorem \ref{theo:removable}, there exists a sequence of times $s_i\ri +\infty$ such that  $\td \Si_{s_i}$ converges smoothly to a plane passing through the origin with multiplicity one. Consider the heat-kernel-type function
$$\Phi_{(x_0, T)}(x, t)=\frac 1{4\pi(T-t)} e^{-\frac {|x-x_0|^2}{4(T-t)}},\quad
\forall\; (x, t)\in \Si_t\times [0, T).$$
Huisken's monotonicity formula(c.f. Theorem 3.1 in \cite{[Hui2]}) implies that
\beqs
\Te(\Si_t,  x_0, T):=\lim_{t\ri T}\,\int_{\Si_t}\, \Phi_{(x_0, T)}(x, t)\,d\mu_t
= \lim_{s_i\ri +\infty} \frac 1{4\pi}\int_{\td \Si_{s_i}}\, e^{-\frac {| x|^2}{4}}\,d\td \mu_{s_i}=1,
\eeqs
which implies that $(x_0, T)$ is a regular point by Theorem 3.1 of B. White \cite{[White2]}.
Thus, the unnormalized mean curvature flow $\{(\Si, \x(t)), 0\leq t<T\}$ cannot blow up at $(x_0, T)$, which contradicts our assumption.
The theorem is proved.

\end{proof}

\begin{appendices}

\section{The parabolic Harnack inequality}
In this appendix, we include the parabolic Harnack inequality from Krylov-Safonov \cite{[Kry]}.
First, we introduce some notations. Let  $x=(x^1, x^2, \cdots, x^n)\in \RR^n$. Denote
\beqs |x|&=&\Big(\sum_{i=1}^n(x^i)^2\Big)^{\frac 12},\quad B_R(x)=\{y\in \RR^n\;|\; |x-y|<R\},\\
Q(\te, R)&=&B_R(0)\times (0, \te R^2).\eeqs
Consider the parabolic operator
\beq
L u=-\pd ut+a^{ij}(x, t)u_{ij}+b^i(x, t)u_i-c(x, t)u, \label{eqB:104}
\eeq where the coefficients are measurable and satisfy the conditions
\beqn
\mu |\xi|^2&\leq &a^{ij}(x, t)\xi_i\xi_j\leq \frac 1{\mu}|\xi|^2,\label{eqB:001}\\
|b(x, t)|&\leq &\frac 1{\mu},\label{eqB:002}\\
0&\leq &c(x, t)\leq \frac 1{\mu}.\label{eqB:003}
\eeqn Here $b(x, t)=(b^1(x, t), \cdots, b^n(x, t)). $
Then we have

\begin{theo}\label{theo:B1} (Theorem 1.1 of \cite{[Kry]}) Suppose the operator $L$ in (\ref{eqB:104}) satisfies the conditions (\ref{eqB:001})-(\ref{eqB:003}).
 Let $\te>1, R\leq 2, u\in W^{1, 2}_{n+1}(Q(\te, R)), u\geq 0$ in $Q(\te, R)$, and $Lu=0$ on $Q(\te, R)$. Then there exists a constant $C$, depending only on $\te, \mu$ and $n$, such that
\beq
u(0, R^2)\leq C\, u(x, \te R^2),\quad \forall\; x\in B_{\frac R2}(0).
\eeq Moreover, when $\frac 1{\te-1}$ and $\frac 1{\mu}$ vary within finite
bounds, $C$ also varies within finite bounds.
\end{theo}

Note that in our case the equation (\ref{eq:F007}) doesn't satisfy the assumption that $c(x, t)\geq 0$ in (\ref{eqB:003}). Therefore, we cannot use Theorem \ref{theo:B1} directly.
However, the Harnack inequality still works when $c(x, t)$ is a constant. Namely, we have

\begin{theo}\label{theo:B2}Let $\te>1, R\leq 2$. Suppose that $u(x, t)\in W^{1, 2}_{n+1}(Q(\te, R))$ is  a nonnegative solution to the equation
\beq
L u=-\pd ut+a^{ij}(x, t)u_{ij}+b^i(x, t)u_i+c u=0,  \label{eqB:004}
\eeq where $c$ is a constant and the coefficients satisfy (\ref{eqB:001})-(\ref{eqB:002}). Then there exists a constant $C$, depending only on $\te, \mu, c$ and $n$, such that
\beq
u(0, R^2)\leq C\, u(x, \te R^2),\quad \forall\; |x|<\frac 12 R.
\eeq
\end{theo}
\begin{proof}Since $u(x, t)$ is a solution of (\ref{eqB:004}) and $c$ is a constant, the function $v(x, t)=e^{-ct}u$ satisfies  \beq
-\pd vt+a^{ij}(x, t)v_{ij}+b^i(x, t)v_i=0. \label{eq:B100}
\eeq
Applying Theorem \ref{theo:B1} to the equation (\ref{eq:B100}), we have $$v(0, R^2)\leq C\, v(x, \te R^2),\quad \forall\; |x|<\frac 12 R, $$
where $C$ depends only on $\te, \mu$ and $n$. Thus, for any $x\in B_{\frac R2}(0)$ we have
$$u(0, R^2)\leq Ce^{-c(\te-1)R^2}u(x, \te R^2)\leq C'u(x, \te R^2),$$
where $C'$ depends only on $\te, \mu, c$ and $n$. Here we used $R\leq 2$  by the assumption. The theorem is proved.

\end{proof}

We generalize Theorem \ref{theo:B2} to a general bounded domain in $\RR^n.$

\begin{theo}\label{theo:B4} Let $\Om$ be a bounded domain in $\RR^n$. Suppose that $u(x, t)\in W^{1, 2}_{n+1}(\Om\times (0, T))$ is  a nonnegative solution to the equation
\beq
L u=-\pd ut+a^{ij}(x, t)u_{ij}+b^i(x, t)u_i+c u=0,  \label{eqB:0041}
\eeq where $c$ is a constant and the coefficients satisfy (\ref{eqB:001})-(\ref{eqB:002}) for a constant $\mu>0$. For any $s, t$ satisfying $0<s<t<T$ and any $x, y\in \Om$ with the following properties
\begin{enumerate}
  \item[(1).] $x$ and $y$ can be connected by a line segment $\ga$ with the length $|x-y|\leq l;$
  \item[(2).] Each point in $\ga$ has a positive distance at least $\dd>0$ from the boundary of $\Om;$
\item[(3)] $s$ and $ t$ satisfy $T_1\leq t-s\leq T_2$ for some $T_1, T_2>0;$
\end{enumerate}
 we have
\beq
u(y, s)\leq C \,u(x, t) , \label{eq:J004}
\eeq where $C$ depends only on $c, n, \mu, \min\{s, \dd^2\}, l, T_1$ and $T_2$.
\end{theo}
\begin{proof}
Let $\ga$ be the line segment with the property $(1)$ and $(2)$ connecting $x$ and $y$. We set $$p_0=y,\quad  p_N=x, \quad p_i=p_0+\frac {x-y}{N}i\in \ga$$ for any $0\leq i\leq N.$ Here we choose $N$ to be the smallest integer satisfying
\beq
N>\max\Big\{\frac {2(t-s)}{s}, \frac {l}{\min\{\frac {\sqrt{s}}{4}, \frac {\dd}4 \}}\Big\}. \label{eqB:100}
\eeq We define
\beq
R=\frac {2l}{N}, \quad  \te=1+\frac {t-s}{R^2N}.\label{eqB:101}
\eeq We can check that $R\leq \frac {\dd}2.$
For any $s, t\in (0, T)$, we choose $\{t_i\}_{i=0}^N$ such that $t_0=s, t_N=t$ and \beq t_{i}-t_{i-1}=\frac {t-s}N \label{eqB:105}\eeq for all integers $1\leq i\leq N$. Note that (\ref{eqB:100})-(\ref{eqB:105}) imply that for any $0\leq i\leq N-1$,
$$t_{i+1}-\te R^2\geq s-\te R^2=s-R^2-\frac {t-s}N\geq \frac s4>0$$
and
$$|p_{i+1}-p_i|=\frac {|x-y|}{N}\leq \frac lN=\frac R2.$$
Therefore, for any $0\leq i\leq N-1$ we have  $(t_{i+1}-\te R^2, t_{i+1})\subset (0, T) $ and $p_{i+1}\in B_{\frac R2}(p_i)$.  Applying Theorem \ref{theo:B2} on $B_R(p_i)\times (t_{i+1}-\te R^2, t_{i+1})\subset \Omega\times (0, T)$, we have
\beq
u(p_i, t_i)\leq C\,u(p_{i+1}, t_{i+1}),
\eeq where $C$ depends only on $c, n, \mu$ and $\frac 1{\te-1}=\frac {R^2N}{t-s}. $ Here we used the fact that
$t_i=(t_{i+1}-\te R^2)+R^2.$
Therefore,
\beq
u(y, s)=u(p_0, t_0)\leq C^Nu(p_N, t_N)=C' u(x, t)  \label{eqB:103}
\eeq
where the constant
$C'$ in (\ref{eqB:103}) depends only on $c, n, \mu, \min\{s, \dd^2\}, l,  T_1$ and $T_2$. The theorem is proved.

\end{proof}

A direct corollary of Theorem \ref{theo:B4} is the following result.

\begin{theo}\label{theo:B3} Let $\Om$ be a bounded domain in $\RR^n$. Suppose that $u(x, t)\in W^{1, 2}_{n+1}(\Om\times (0, T))$ is  a nonnegative solution to the equation
\beq
L u=-\pd ut+a^{ij}(x, t)u_{ij}+b^i(x, t)u_i+c u=0,  \label{eqB:0042}
\eeq where $c$ is a constant and the coefficients satisfy (\ref{eqB:001})-(\ref{eqB:002}) for a constant $\mu>0$, and $\Om', \Om''$ are subdomains in $\Om$ satisfying the following properties:
\begin{enumerate}
  \item[(1).]$\Om'\subset \Om''\subset \Om$, and $\Om''$ has a positive distance $\dd>0$ from the boundary of $\,\Om$;

  \item[(2).] $\Om'$ can be covered by $k$ balls with radius $r$, and all  balls are contained in $\Om''.$
\end{enumerate}
Then for any $s, t$ satisfying $0<s<t<T$ and any $x, y\in \Om'$,
 we have
\beq
u(y, s)\leq C \,u(x, t) , \label{eq:J0041b}
\eeq where $C$ depends only on $c, n, \mu, \min\{s, \dd^2\}, t-s, r$ and $k$.
\end{theo}
\begin{proof}By the assumption, we can find finite many points $\cA=\{q_1, q_2, \cdots, q_k\}$ such that
\beq
\Om'\subset \cup_{q\in \cA}B_r(q)\subset \Om''.
\eeq
For any $x, y\in \Om'$, there exists two points in $\cA$, which we denote by $q_1$ and $q_2$, such that $x\in B_{r}(q_1)$ and $y\in B_{r}(q_2)$. Then $x$ and $y$ can be connected by a polygonal chain $\ga$,
which consists of two line segments $\overline{xq_1}, \overline{yq_2}$ and a polygonal chain   with vertices in $\cA$ connecting $q_1$ and $q_2$. Clearly, the number of the vertices of $\ga$ is bounded by $k+2$
and the total length of $\ga$ is bounded by $(k+2) r.$ Moreover, by the assumption we have $\ga\subset \Om''$ and  each point in $\ga$ has a positive distance at least $\dd>0$ from the boundary of $\Om.$

 Assume that the polygonal chain $\ga$ has consecutive  vertices $\{p_0, p_1, \cdots, p_N\}$ with $p_0=y, p_N=x$ and $1\leq N\leq k+2$.  We apply Theorem \ref{theo:B4} for each line segment $\overline{p_ip_{i+1}}$ and the interval $[t_i, t_{i+1}]$, where $\{t_i\}$ is chosen as in (\ref{eqB:105}). Note that
$$\frac {t-s}{k+2}\leq t_{i+1}-t_i=\frac {t-s}{N}\leq t-s.$$
Thus,  for any $0\leq i\leq N-1$ we have
\beq
u(p_i, t_i)\leq Cu(p_{i+1}, t_{i+1}),\label{eqC:003}
\eeq where $C$ depends only on $c, n, \mu, \min\{s, \dd^2\}, r, k$ and $t-s$, and (\ref{eqC:003}) implies (\ref{eq:J0041b}).  This finishes the proof of Theorem \ref{theo:B3}.

\end{proof}

\section{The interior estimates of parabolic equations}
In this appendix, we present the interior estimates of parabolic equations from G. Lieberman's book \cite{[Lieb]} for the reader's convenience.

  Let $X=(x, t)$ be a point in $\RR^{n+1}$ and $x=(x_1, x_2, \cdots, x_n)\in \RR^n$. The norms on $\RR^n$ and $\RR^{n+1}$ are given by
$$|x|=\Big(\sum_{i=1}^n(x^i)^2\Big)^{\frac 12},\quad |X|=\max\{|x|, |t|^{\frac 12}\}. $$
Let $\Om$ be a domain in $\RR^{n+1}$.
Let $d(X, Y)=\min\{d(X), d(Y)\}$ where $d(X)=\mathrm{dist}(X, \cP\Om\cap \{t<t_0\})$. Here $\cP \Om$ denotes the parabolic boundary of $\Om$.  We define
$$|f|_0=\sup_{\Om} |f|. $$
If $b\geq 0$, we define
$$|f|_0^{(b)}=\sup_{X\in \Om}d(X)^b|f(X)|. $$
If $a=k+\al>0$ and $a+b\geq 0$, where $k$ is a nonnegative integer and $\al\in (0, 1]$, we define
\beqs
[f]_a^{(b)}&=&\sup\Big\{\sum_{|\bb|+2j=k}d(X, Y)^{a+b}\frac {|D_x^{\bb}D_t^j f(X)-D_x^{\bb}D_t^j f(Y)|}{|X-Y|^{\al}}: X\neq Y \; \hbox{in}\;\Om\Big\},\\
\langle f\rangle_a^{(b)}&=&\sup\Big\{\sum_{|\bb|+2j=k-1}d(X, Y)^{a+b}\frac {|D_x^{\bb}D_t^j f(X)-D_x^{\bb}D_t^j f(Y)|}{|X-Y|^{1+\al}}: X\neq Y \; \hbox{in}\;\Om,\;x=y\Big\}, \\
|f|_a^{(b)}&=& \sum_{|\bb|+2j\leq k}|D_x^{\bb}D_t^jf|_0^{(|\bb|+2j+b)}+[f]_a^{(b)}+\langle f\rangle_a^{(b)}.
\eeqs
We also define $|f|_a^*=|f|_a^{(0)}$, and define the spaces
$$H_a^*=\{f:\;|f|_a^*<\infty\},\quad H_a^{(b)}=\{f:\; |f|_a^{(b)}<\infty\}. $$

With these notations, we have the following result.

\begin{theo}\label{theo:C1}(Theorem 4.9 of \cite{[Lieb]}) Let $\Om$ be a bounded domain in $\RR^{n+1}$, and let $a^{ij}\in H_{\al}^{(0)}$ and $b^i\in H_{\al}^{(1)}$ satisfy
\beqn
\la |\xi|^2&\leq& a^{ij}\xi_i\xi_j\leq \La |\xi|^2, \quad [a^{ij}]_{\al}^{(0)}\leq A, \\
|b^i|_{\al}^{(1)}&\leq &B
\eeqn for some constants $A, B, \la$ and $\La$. Let $c\in H_{\al}^{(2)}$ satisfy
\beq
|c|_{\al}^{(2)}\leq c_1
\eeq for some constant $c_1$ and let $f\in H_{\al}^{(2)}$. If $u\in H^*_{2+\al}$ is a solution of
\beq
-\pd ut+a^{ij}u_{ij}+b^i u_i+c u=f
\eeq in $\Om$, then there is a constant $C$ determined only by $A, B, c_1, n, \la$ and $\La$ such that
\beq
|u|_{2+\al}^*\leq C(|u|_0+|f|_{\al}^{(2)}).
\eeq

\end{theo}

\end{appendices}

\vskip10pt
Haozhao Li,  Key Laboratory of Wu Wen-Tsun Mathematics, Chinese Academy of Sciences,  School of Mathematical Sciences, University of Science and Technology of China, No. 96 Jinzhai Road, Hefei, Anhui Province, 230026, China;  hzli@ustc.edu.cn.\\

Bing  Wang,  School of Mathematical Sciences, University of Science and Technology of China, No. 96 Jinzhai Road, Hefei, Anhui Province, 230026, China;
Department of Mathematics, University of Wisconsin-Madison,  Madison, WI 53706, USA;  bwang@math.wisc.edu.\\

\end{document}